\documentclass[12pt]{amsart}
\usepackage{amssymb,amsmath,amsfonts,latexsym}
\usepackage[all,cmtip]{xy}

\newcommand{\newsection}[1]{\setcounter{equation}{0} \section{#1}}

\newcommand{\vp}{\varphi}
\newcommand{\cla}{\mathcal{A}}
\newcommand{\clb}{\mathcal{B}}
\newcommand{\cld}{\mathcal{D}}
\newcommand{\cle}{\mathcal{E}}
\newcommand{\clf}{\mathcal{F}}
\newcommand{\clh}{\mathcal{H}}

\newcommand{\cli}{\mathcal{I}}
\newcommand{\clm}{\mathcal{M}}

\newcommand{\clp}{\mathcal{P}}
\newcommand{\clq}{\mathcal{Q}}

\newcommand{\cls}{\mathcal{S}}
\newcommand{\clw}{\mathcal{W}}
\newcommand{\clz}{\mathcal{Z}}

\newcommand\Zp{\mathbb{Z}_+^n}
\newcommand{\B}{\mathbb{B}}
\newcommand{\D}{\mathbb{D}}
\newcommand{\N}{\mathbb{N}}
\newcommand{\C}{\mathbb{C}}
\newcommand{\Z}{\mathbb{Z}}
\newcommand{\T}{\mathbb{T}}
\newcommand{\Sn}{\mathbb{S}^n}
\newcommand{\Sk}{K_n}
\newcommand{\raro}{\rightarrow}

\newcommand{\be}{\begin{equation}}
\newcommand{\ee}{\end{equation}}

\newcommand\la{{\langle }}
\newcommand\ra{{\rangle}}

\newtheorem{theorem}{Theorem}[section]
\newtheorem{lemma}[theorem]{Lemma}
\newtheorem{proposition}[theorem]{Proposition}
\newtheorem{corollary}[theorem]{Corollary}
\theoremstyle{definition}
\newtheorem{definition}[theorem]{Definition}
\newtheorem{example}[theorem]{Example}
\newtheorem{problem}{Problem}

\newtheorem{remark}[theorem]{Remark}
\numberwithin{equation}{section}
\newtheorem{thm}{Theorem}

\begin{document}
\title{Approximation, interpolation, and lifting on the unit ball}

\author[Bhattacharjee]{Jaydeep Bhattacharjee}
\address{Indian Statistical Institute, Statistics and Mathematics Unit, 8th Mile, Mysore Road, Bangalore, 560059,
India}
\email{jaydeepbhattacharjee98@gmail.com}

\author[Deepak]{Deepak K. D.}
\address{Department of Applied Mathematics, National Yang Ming Chiao Tung University, Hsinchu City, 30010, Taiwan}
\email{dpk.dkd@gmail.com }

\author[Sarkar]{Jaydeb Sarkar}
\address{Indian Statistical Institute, Statistics and Mathematics Unit, 8th Mile, Mysore Road, Bangalore, 560059,
India}
\email{jay@isibang.ac.in, jaydeb@gmail.com}


\subjclass[2010]{30E05, 46J15, 47A57, 30J05, 32A40, 32A65, 42B30, 15B05, 30H10, 33F05, 30J15, 43A90, 41A05, 42B15}

\keywords{Approximation, extremal functions, Blaschke products, inner functions, Toeplitz operators, perturbations, commutant lifting, Hardy space, unit ball, analytic functions,  interpolation}

\begin{abstract}
We solve the Nevanlinna-Pick interpolation problem on the open unit ball of the complex $n$-space. Our solutions signify the role of inner functions on the unit ball, objects whose existence was once considered uncertain. The results also reveal the importance of extremal functions, which emerge as natural analogues of finite Blaschke products. This viewpoint is illustrated by the Carath\'{e}odory approximation and Pick's theorems on the unit ball. We solve the commutant lifting problem, where both inner and extremal functions play a fundamental role. These results resolve several well-known problems on the unit ball.
\end{abstract}

\maketitle

\tableofcontents

\newsection{Introduction }\label{sec: introd}

For each $n \geq 1$, let $\B^n$ denote the open unit ball in $\C^n$:
\[
\B^n = \left\{z = (z_1, \ldots, z_n) \in \mathbb{C}^n: \sum_{i=1}^{n} |z_i|^2 < 1\right\}.
\]
We write $\mathbb{B}^1 = \mathbb{D}$. Let $H^\infty(\B^n)$ denote the space of all scalar-valued bounded analytic functions on $\B^n$. Equipped with the supremum norm $\|\cdot\|_\infty$, this space forms a commutative Banach algebra. Let the closed unit ball of $H^\infty(\B^n)$ be represented as $\cls(\B^n)$. Therefore,
\[
\cls(\B^n) = \{\vp \in H^\infty(\B^n): \|\vp\|_\infty \leq 1\}.
\]
Recall that $\|\vp\|_\infty := \displaystyle\sup_{z \in \B^n}|\vp(z)|$. The elements of $\cls(\B^n)$ are called \textit{Schur functions}.

By \textit{interpolation data} (or simply \textit{data}), we mean a pair $(\clz, \clw)$, where $\clz= \{z_i\}_{i=1}^m$ is a set of $m$ distinct points in $\B^n$, called the interpolation nodes, and $\clw= \{w_i\}_{i=1}^m$ is a collection of $m$ scalars in $\D$, called the target values. We simply write the data as $\clz$ and $\clw$.

\begin{problem}[Nevanlinna–Pick interpolation problem]\label{thm intro NP}
Given data $\clz = \{z_i\}_{i=1}^m \subseteq \mathbb{B}^n$ and $\clw = \{w_i\}_{i=1}^m \subseteq \mathbb{D}$, determine when there exists a function $\varphi \in \mathcal{S}(\mathbb{B}^n)$ such that
\[
\varphi(z_i) = w_i,
\]
for all $i=1, \ldots, m$.
\end{problem}

If such a function $\varphi \in \cls(\B^n)$ exists, then we say that the interpolation data $\clz$ and $\clw$ solve the \textit{Nevanlinna–Pick interpolation problem}, or simply \textit{the interpolation problem}. In this case, $\vp$ is called an \textit{interpolant}, and we say that $\vp$ \textit{interpolates} $\clz$ and $\clw$. 

The extremal problem imposes additional restrictions on the interpolants. We say that interpolation data $\clz = \{z_i\}_{i=1}^m \subseteq \mathbb{B}^n$ and $\clw = \{w_i\}_{i=1}^m \subseteq \mathbb{D}$ solve the \textit{extremal problem} if the data solve the interpolation problem with the property that
\[
\|\vp\|_\infty=1,
\]
for every $\varphi \in \cls(\B^n)$ that interpolates $\clz$ and $\clw$

\begin{problem}[Extremal problem]\label{thm intro EP}
Determine when a given data $\clz = \{z_i\}_{i=1}^m \subseteq \mathbb{B}^n$ and $\clw = \{w_i\}_{i=1}^m \subseteq \mathbb{D}$ solve the extremal problem.
\end{problem}

We collect all solutions to the extremal problem:
\[
\cle(\B^n) = \{\vp \in \cls(\B^n): \vp \text{ solves some extremal problem}\}.
\]
Functions in this class are called \textit{extremal functions}. Therefore, an extremal function is a solution to an extremal problem corresponding to some data $\clz = \{z_i\}_{i=1}^m \subseteq \mathbb{B}^n$ and $\clw = \{w_i\}_{i=1}^m \subseteq \mathbb{D}$ for some $m>1$. We emphasize that the one-point interpolation problem (that is, when $m=1$) can never be extremal.

The interpolation problems described above have long been of central importance in several areas of mathematics and certain branches of engineering. The problem of characterizing interpolation on $\B^n$ has remained open for many years, and in this paper, we provide a complete solution. Our approach, however, goes considerably further and solves many other well-known problems. We establish the significance of extremal functions on $\B^n$ and essentially prove that they play the role of finite Blaschke products and develop an approximation theory that, in particular, extends the classical Carath\'{e}odory theorem from $\D$ to $\B^n$. This is developed, in particular, as a consequence of our solution to the Pick problem on $\B^n$. We show that inner functions provide a powerful tool in interpolation and play a fundamental role in our development of the theory. Their importance in one-variable function theory is well known; however, on $\B^n$, even their existence was uncertain until the pioneering work of Aleksandrov \cite{Alek} (also see Rudin \cite{Rud new fn}).

We briefly review the classical results on interpolation problems. When $n = 1$, Problem \ref{thm intro NP} was solved independently by Nevanlinna \cite{Nevanlinna} and Pick \cite{G Pick} in the 1910s:

\begin{theorem}[Nevanlinna–Pick interpolation theorem on $\D$]
There exists $\vp \in \cls(\D)$ such that $\varphi(z_i) = w_i$ for all $i=1, \ldots, m$, if and only if the $m \times m$ \textit{Pick matrix}
\[
\Big[\frac{1 - w_i \overline{w_j}}{1 - z_i \overline{z_j}}\Big]_{m \times m},
\]
is positive semi-definite.
\end{theorem}

We now turn to the solution of Problem \ref{thm intro EP} in the classical case $n=1$. The interesting aspect is its connection to finite Blaschke products (cf. \cite[page 9, Corollary 2.3]{Garnett}):

\begin{theorem}[Extremal problem on $\D$]\label{thm E(D) = finite BP}
Interpolation data $\clz = \{z_i\}_{i=1}^m \subseteq \D$ and $\clw = \{w_i\}_{i=1}^m \subseteq \mathbb{D}$ solve the extremal problem if and only if the solution to the interpolation problem is unique. In this case, the unique solution is a finite Blaschke product. In particular,
\[
\cle(\D) = \{\text{finite Blaschke product}\}.
\]
\end{theorem}

Recall that a \textit{finite Blaschke product} is a function $b: \D \raro \D$ of the form
\[
b(z) = \lambda z^p \displaystyle\prod_{j=1}^q \frac{\overline{\alpha_j}}{|\alpha_j|} \frac{z - \alpha_j}{1 - \overline{\alpha}_j z},
\]
where $\lambda \in \T (= \partial \D)$, $\alpha_1, \ldots, \alpha_q \in \D$, and $p,q \in \Z_+$. Finite Blaschke products are essential in function theory. We highlight two results that illustrate their significance and are particularly relevant in our several variables context. The first is a classical result of Pick \cite{G Pick}:

\begin{theorem}[Pick's theorem on $\D$]\label{thm: Intro Pick}
Suppose the data $\clz = \{z_i\}_{i=1}^m \subseteq \D$ and $\clw = \{w_i\}_{i=1}^m \subseteq \mathbb{D}$ solve the interpolation problem. Then there exists
\[
b \in \cle(\D),
\]
that interpolates $\clz$ and $\clw$.
\end{theorem}

By Theorem \ref{thm E(D) = finite BP}, we know that $b$ is a finite Blaschke product. The second is a classical result of Carath\'{e}odory \cite{Carat}:

\begin{theorem}[Carath\'{e}odory's theorem on $\D$]\label{thm: Intro Carath}
Let $\vp \in \cls(\D)$. Then there exists a sequence $\{b_k\}_k \subseteq \cle(\D)$ such that
\[
\lim_{k \raro \infty}{b_k} = \vp,
\]
uniformly on compact subsets of $\D$.
\end{theorem}

Again, Theorem \ref{thm E(D) = finite BP} implies that $\{b_k\}_k$ is a sequence of finite Blaschke products.

A closely related result to the above Carath\'{e}odory theorem concerns inner functions on $\B^n$, particularly in the case $n>1$. Let $\Sn = \partial \B^n$. A function $\vp \in H^{\infty}(\B^n)$ is said to be \textit{inner} if
\[
|\vp(\zeta)| = 1,
\]
for almost every $\zeta \in \Sn$ with respect to the surface measure, where the boundary values are taken in the sense of $K$-limits (see \cite{Rud} or Section \ref{sec: quality char}). We write
\[
\cli(\B^n) = \{\vp \in H^{\infty}(\B^n): \vp \text{ is inner}\}.
\]
In one variable, finite Blaschke products are the simplest examples of inner functions. In contrast, on $\B^n$, $n > 1$, there is no comparable description or understanding of finite Blaschke products, let alone inner functions in general. Indeed, for a long time, it was widely conjectured that no nonconstant inner functions exist on $\B^n$, $n > 1$. This belief persisted until 1981, when Aleksandrov \cite{Alek} exhibited the existence of such functions (they are pathological); shortly thereafter, L\"{o}w \cite{L} independently reestablished their existence (also see Hakim and Sibony \cite{HS}; and see Rudin’s historical comments in \cite{Rud inner}). Even more surprisingly, despite their pathological nature, such functions exist in enormous abundance \cite{Alek, Rud inner}:

\begin{theorem}[Aleksandrov-Rudin theorem on $\B^n$]\label{thm: intro Alek Rud}
Let $\vp \in \cls(\B^n)$. Then there exists a sequence of inner functions $\{\vp_k\}_k \subseteq \cli(\B^n)$ such that
\[
\lim_{k \raro \infty}{\vp_k} = \vp,
\]
uniformly on compact subsets of $\B^n$, and
\[
\lim_{k \raro \infty} \int_{\Sn} \vp_k f \,d \sigma = \int_{\Sn} \vp f \,d \sigma,
\]
for all $f \in L^1(\Sn)$, where $d\sigma$ denotes the normalized surface measure on $\Sn$.
\end{theorem}

In contrast to the scarcity of concrete examples of inner functions, there is an abundance of extremal functions \cite{K}. We view extremal functions as the natural analogues of finite Blaschke products on $\B^n$. Indeed, when $n=1$, as pointed out in Theorem \ref{thm E(D) = finite BP}, these two classes coincide. From this perspective, we establish the following generalization of Pick's result from $\D$ to $\B^n$ (see Theorem \ref{thm: extr sol}). In what follows, unless otherwise stated, the case of $n = 1$ will also be included in the discussion.

\begin{thm}\label{thm: intro extr int}
If the data $\clz = \{z_i\}_{i=1}^m \subseteq \B^n$ and $\clw = \{w_i\}_{i=1}^m \subseteq \D$ solve the interpolation problem, then there exists an extremal function
\[
\vp \in \cle(\B^n),
\]
that interpolates the data $\clz$ and $\clw$.
\end{thm}

As an application of this result, we establish the following extension of Carath\'{e}odory's theorem from $\D$ to $\B^n$ (see Theorem \ref{thm:Car sv}):

\begin{thm}
Let $\vp \in \cls(\B^n)$. Then there is a sequence $\{\vp_k\}_k \subseteq \cle(\B^n)$ satisfying
\[
\lim_{k \raro \infty} {\vp_k} = \vp,
\]
uniformly on compact subsets of $\B^n$. Furthermore,
\[
\lim_{k \raro \infty} \int_{\Sn} \vp_k f \,d \sigma = \int_{\Sn} \vp f \,d \sigma,
\]
for all $f \in L^1(\Sn)$.
\end{thm}

Note that this result is also in the spirit of the Aleksandrov-Rudin theorem on $\B^n$, with the class of inner functions replaced by the class of extremal functions. The Aleksandrov-Rudin theorem, together with the above yields:
\[
\cls(\B^n) = \overline{\cli(\B^n)}^{(w^*, L^\infty(\Sn))} = \; \overline{\cle(\B^n)}^{(w^*, L^\infty(\Sn))}.
\]
Therefore, we now have two distinguished classes of functions that approximate Schur functions: $\cli(\B^n)$ and $\cle(\B^n)$. Theorem \ref{thm: intro extr int} highlights the role of $\cle(\B^n)$ in the interpolation problem, while the following result does the same for $\cli(\B^n)$ (see Theorem \ref{thm: NP interpolation inner}):

\begin{thm}
Let $\clz = \{z_i\}_{i=1}^m \subseteq \B^n$ and $\clw = \{w_i\}_{i=1}^m \subseteq \D$ be interpolation data. Then $\clz$ and $\clw$ solve the interpolation problem if and only if there exists a sequence of inner functions $\{\vp_k\}_k \subseteq \cli(\B^n)$ such that for each $j=1, \ldots, m$,
\[
\lim_{k \raro \infty} \vp_k(z_j) = w_j.
\]
\end{thm}

We remark that several results, examples, and counterexamples in this paper complement the theory of inner functions developed by Aleksandrov \cite{Alek} and Rudin \cite{Rud inner}. Extremal functions play a similarly important role throughout the paper.

The significance of the preceding result lies in the way it highlights the inherent complexity of the interpolation problem and its deep connection to the subtleties of function theory on $\B^n$, $n > 1$. In addition to the above, we present several other characterizations of the interpolation problem, which are collected in Theorem \ref{thm intro interp}.

One approach to our solution of the interpolation problem follows in the footsteps of Sarason \cite{D Sarason}: we first establish a commutant lifting theorem for $H^2(\B^n)$, the Hardy space over $\B^n$, and then use it to resolve the interpolation problem. We further apply this to characterize the extremal problem. It should be noted that the commutant lifting theorem for $H^2(\B^n)$, in the case $n>1$, is itself an important and challenging problem. Recall that $H^2(\B^n)$ consists of all analytic functions $f: \B^n \rightarrow \C$ satisfying
\[
\|f\|:= \Big( \sup_{0< r < 1} \int_{\Sn} |f(r \zeta_1, \ldots, r \zeta_n)|^2 \,d\sigma \Big)^{\frac{1}{2}} < \infty.
\]
Each $\vp \in H^\infty(\mathbb{B}^n)$ defines a multiplication operator $T_\vp$ on $H^2(\B^n)$ given by
\[
T_\vp f = \vp f \qquad (f \in H^2(\B^n)),
\]
with the property that $\|T_\vp\| = \|\vp\|_\infty$. The operator $T_\vp$ is also known as the \textit{analytic Toeplitz operator} with symbol $\vp$. The coordinate functions $\varphi = z_i$ give rise to the special multiplication operators $T_{z_i}$, $i=1, \ldots, n$. We have the fundamental commutation property:
\[
\{T_{z_1}, \ldots, T_{z_n}\}' = \{T_\vp: \vp \in H^\infty(\B^n)\}.
\]
A closed subspace $\clq$ of $H^2(\B^n)$ is called a \textit{quotient module} \cite{CG, DP, GZ} if for each $i=1, \ldots, n$,
\[
T_{z_i}^* \clq \subseteq \clq,
\]
Given a Hilbert space $\clh$, we denote by $\clb(\clh)$ the space of all bounded linear operators on $\clh$, and we set
\[
\clb_1(\clh) = \{T \in \clb(\clh): \|T\| \leq 1\}.
\]
Given a quotient module $\clq \subseteq H^2(\B^n)$ and a function $\varphi \in H^\infty(\B^n)$, we define $S_\vp \in \clb_1(\clq)$ by
\[
S_\vp = P_\clq T_\vp|_{\clq},
\]
where $P_\clq$ denotes the orthogonal projection of $H^2(\B^n)$ onto $\clq$ \cite{DP}. We always assume that
\[
\clq \neq \{0\}.
\]
Of particular interest are the operators $S_{z_i} \in \clb_1(\clq)$, where
\[
S_{z_i} f = P_\clq (z_i f),
\]
for all $f \in \clq$ and $i=1, \ldots, n$. We say that an operator $X \in \clb(\clq)$ is a \textit{module map} if
\[
X S_{z_i} = S_{z_i} X,
\]
for all $i=1, \ldots, n$.
With this terminology and structure in place, we can now formulate the commutant lifting problem as follows:

\begin{problem}\label{prob: CLT}
Let $\clq$ be a quotient module of $H^2(\B^n)$, and let $X \in \clb_1(\clq)$ be a module map. When does there exist a Schur function $\vp \in \cls(\B^n)$ such that
\[
X = S_\vp?
\]
\end{problem}

Equivalently, this is a question about the commutativity of the following diagram, together with the additional requirement that $\vp \in \cls(\B^n)$:
\[
\xymatrix{
H^2(\B^n) \ar@{->}[rr]^{\displaystyle T_{\vp}} \ar@{<-}[dd]_{\displaystyle i_{\clq}}
&& H^2(\B^n) \ar@{->}[dd]^{\displaystyle i^*_{\clq}}    \\ \\
\clq \ar@{->}[rr]_{\displaystyle X} && \clq
}
\]
where $i_\clq : \clq \longrightarrow H^2(\B^n)$ denotes the inclusion map (here we note that $S_\vp = i_\clq^* T_\vp i_\clq$). If such an operator $X \in \clb_1(\clq)$ admits a representation $X = S_\vp$ for some $\vp \in \cls(\B^n)$, then we say that $X$ admits a \textit{lift}, or that $X$ admits a \textit{lift to $T_\vp$}, or simply that $X$ is \textit{liftable}.

When $n = 1$, Sarason's seminal work asserts that every such module map $X \in \clb_1(\clq)$ admits a lift \cite{D Sarason}. Moreover, $X$ admits a \textit{norm preserving} lift:

\begin{definition}
Let $\clq \subseteq H^2(\B^n)$ be a quotient module. A module map $X \in \clb_1(\clq)$ is said to admit a norm preserving lift if there exists $\varphi \in \cls(\B^n)$ such that $X = S_\varphi$ and
\[
\|X\| = \|\varphi\|_\infty.
\]
In this case, $S_\vp$ is said to be a norm-preserving lift of $X$.
\end{definition}

Sarason applied his lifting theorem to certain finite-dimensional quotient modules, thereby recovering the Nevanlinna–Pick interpolation theorem. Now we turn to several variables.

First we establish the essential tools that would be needed for the lifting theorem. In view of $K$-limits, we shall, when convenient, identify $H^2(\B^n)$ with $H^2(\Sn)$ without further explanation (cf. Section \ref{sec: quality char}), where
\[
H^2(\Sn) = \overline{\mathbb{C}[z_1, \ldots, z_n]}^{L^2(\Sn)}.
\]
For $S \subseteq L^2(\Sn)$, let $S^{conj}$ denote the set of all conjugate functions of elements of $S$:
\[
S^{conj} = \{\overline{f}: f \in S\},
\]
where $\overline{f}$ denotes the conjugate of $f$. Next, we define the space of ``mixed functions'' (also see \cite{KD}):
\[
\clm(\Sn):= L^2(\Sn) \ominus [H^2(\Sn) + H^2(\Sn)^{conj}].
\]
Note that $\clm(\mathbb{T}) = \{0\}$ (recall that $\mathbb{S}^1 = \partial \mathbb{B} = \mathbb{T}$). This space admits a convenient decomposition in terms of harmonic homogeneous polynomials \cite[Chapter 12]{Rud}, which will be outlined in Section \ref{sec: remark}. Define the space of Hardy functions ``vanishing at $0$'' as
\[
H^2_0(\Sn) = H^2(\Sn) \ominus \{1\}.
\]
Let $\clq$ be a quotient module of $H^2(\B^n)$, which, in view of $K$-limits, we may regard as a closed subspace of $H^2(\Sn)$, as noted above. We consider $\clq^{conj}$, $\clm(\Sn)$, and $H^2_0(\Sn)$ as subspaces of $L^1(\Sn)$ and define
\[
\clm_{\clq} = \clq^{conj} \dot+ [\clm(\Sn) \dot+ H^2_0(\Sn)].
\]
where $\dot+$ denotes the skew sum in the Banach space $L^1(\Sn)$.

We now return to the inner and extremal functions. Throughout the paper, the \textit{approximate class} $\cla(\B^n)$ denotes either $\cli(\B^n)$ or $\cle(\B^n)$. Let $S \subseteq L^1(\Sn)$, $\varphi \in L^\infty(\Sn)$, and let $\{\vp_k\}$ be a sequence in $\cla(\B^n)$. We say that
\[
w^*-\lim_{k \raro \infty} \vp_k = \vp \text{ on } S,
\]
if for every $f \in S$, we have
\[
\lim_{k \raro \infty} \int_{\Sn} \vp_k f d\sigma = \int_{\Sn} \vp f d\sigma.
\]

The following is a summary of all the commutant lifting theorems obtained in this paper. These results also apply to the case $n = 1$ (but there will be exceptions like Theorem \ref{inner ball}), and hence many of them are new even in the classical setting.

\begin{thm}\label{thm: all CLT}
Let $\clq$ be a quotient module of $H^2(\B^n)$, and let $X \in \clb_1(\clq)$ be a nonzero module map. Set
\[
\psi = X (P_\clq 1).
\]
The following conditions are equivalent:
\begin{enumerate}
\item (Commutant lifting) $X$ admits a lift.
\item (Perturbations) There exists $\vp \in \cls(\B^n)$ such that
\[
\psi - \vp \in \clq^{\perp}.
\]
Moreover, in this case $X = S_{\vp}$.
\item (Qualitative property) $X_{\clq}: (\clm_{\clq}, \|.\|_{1}) \longrightarrow \mathbb{C}$ is a contraction, where
\[
X_{\clq} f = \int_{\Sn} \psi f d\sigma \qquad (f \in \clm_{\clq}).
\]
\item (Quantitative property) $\text{dist}_{L^1(\Sn)}\Big(\frac{\overline \psi}{\|\psi\|_2^2}, \widetilde{\clm_{\clq}} \Big) \geq 1$, where
\[
\widetilde{\clm_{\clq}} = [\clq^{conj} \ominus \{\overline{\psi}\}] \dot+ [\clm(\Sn) \dot+ H^2_0(\Sn)].
\]

\item (Approximate functions - I) There exists a sequence $\{\vp_k\} \subseteq \cla(\B^n)$ such that
\[
w^*-\lim_{k \raro \infty} \vp_k = \psi \text{ on }\clm_\clq.
\]

\item (Approximate functions - II) There exists a sequence $\{\vp_k\} \subseteq \cla(\B^n)$ such that
\[
w^*-\lim_{k \raro \infty} \vp_k = \psi \text{ on } \clq^{conj}.
\]
\end{enumerate}
\end{thm}

In the context of the above theorem, we further point out that Lemma \ref{lemm: ker XQ} yields the following identity:
\[
\ker X_\clq = \widetilde{\clm_{\clq}}.
\]

The equivalence of (1) and (2) is less involved, yet it plays a central role in establishing the remaining equivalence conditions as well as several other results in this paper. Moreover, the perturbation property in (2) is connected, perhaps coincidentally, with Rudin’s perturbation theorem for inner functions \cite{Rud inner}, a connection that we will explore further in Section \ref{Sec: examples}.

The above quantitative and qualitative properties of lifting also appear in the polydisc setting \cite{KD}, which we will elaborate on in Section \ref{sec: remark}.

Following Sarason \cite{D Sarason}, a natural question is whether module maps on quotient modules of $H^2(\B^n)$, $n > 1$, admit norm-preserving liftings. The situation in several variables, however, changes drastically. A function-theoretic result of Rudin concerning the Smirnov class for $n>1$ yields the following somewhat unexpected consequence (see Corollary \ref{cor: norm preserving}):

\begin{thm}
Let $X$ be a nonconstant module map acting on a finite-dimensional quotient module $\clq \subseteq H^2(\B^n)$. If
\[
n > 1,
\]
then $X$ does not admit a norm-preserving lift.
\end{thm}

In fact, we have the following characterization of norm-preserving liftings of module maps on finite-dimensional quotient modules (see Theorem \ref{lemma: Xq X norm} and Corollaries \ref{lemma: Xq less X} and \ref{remark: single variable xq x}):

\begin{thm}
Suppose $n \geq 1$. Let $\clq \subseteq H^2(\B^n)$ be a finite-dimensional quotient module and let $X \in \clb_1(\clq)$ be a nonconstant module map. Consider the functional $X_\clq: (\clm_\clq, \|\cdot\|_1) \raro \mathbb{C}$ as in the qualitative part of Theorem \ref{thm: all CLT}. The following are equivalent:
\begin{enumerate}
\item $X$ admits a norm preserving lift.
\item $\|X\| = \|X_{\clq}\|$.
\item $n=1$.
\end{enumerate}
\end{thm}

A key ingredient in the proof of the above theorem, and indeed of many subsequent results in this paper, is the norm formula for the functional $X_\clq$. Namely, in the setting of Theorem \ref{thm: all CLT}, if $X\in\clb_1(\clq)$ admits a lift, then Theorem \ref{thm: norm XQ} implies
\[
\|X_\clq\| = \inf\{\|\vp\|_\infty: \vp \in \cls(\B^n) \text{ and } X = S_\vp\}.
\]
In other words, $\|X_\clq\|$ equals the minimum possible $\cls(\B^n)$-norm among all symbols representing $X$. This identity is intriguing and shows that the choice of the functional $X_\clq$ is appropriate and holds further potential for related problems.

Now we turn to a detailed presentation of our solution to the interpolation problem. It is important to emphasize that a Pick matrix-type characterization for Schur function interpolants is generally not expected in several variables, since such criteria are intrinsically tied to complete Nevanlinna–Pick kernels \cite{AM 2000, DMS, Mc, PQ}. In contrast, the Szeg\"{o} kernel $K_n$ on $\B^n$, which is central in the Schur function framework, is not a complete Nevanlinna–Pick kernel except for the case of $n=1$. Recall that $\Sk: \B^n \times \B^n \rightarrow \C$ is defined by
\[
\Sk(z, w) = (1 - \langle z, w \rangle)^{-n},
\]
where $\langle z, w \rangle = \sum_{i=1}^{n} z_i \overline{w_i}$ for all $z, w \in \B^n$. For each $w \in \B^n$, the \textit{Szeg\"{o} kernel function} $\Sk(\cdot, w) \in H^2(\B^n)$ is defined by
\[
(\Sk(\cdot, w))(z) = \Sk(z, w),
\]
for all $z \in \B^n$. Given interpolation data $\clz = \{z_i\}_{i = 1}^m \subseteq \B^n$ and $\{w_i\}_{i = 1}^m \subseteq \D$, define $m$-dimensional quotient module $\clq_\clz \subseteq H^2(\B^n)$ by
\[
\clq_\clz = \text{span}\{\Sk(\cdot, z_j): j =1, \ldots, m\}.
\]
Also, define a function
\begin{equation}\label{eqn: sect 0 psi}
\psi_{\clz, \clw} = \sum_{j = 1}^m c_j \Sk(\cdot, z_j).
\end{equation}
where the scalars $\{c_j\}_{j=1}^m$ are given by
\[
\begin{bmatrix}
c_1
\\
\vdots
\\
c_m
\end{bmatrix} =
\begin{bmatrix}
\Sk(z_1, z_1) & \dots & \Sk(z_1, z_m)
\\
\vdots &  \ddots & \vdots
\\
\Sk(z_m, z_1) & \dots & \Sk(z_m, z_m)
\end{bmatrix}^{-1} \begin{bmatrix}
w_1
\\
\vdots\\
w_m
\end{bmatrix}.
\]
The inverse makes sense, since the matrix in question is a Gram matrix. There are two notable features that immediately make this function particularly distinguished:
\begin{enumerate}
\item $\psi_{\clz, \clw}$ interpolates the prescribed data; that is, for all $i=1, \ldots, m$,
\[
\psi_{\clz, \clw}(z_i) = w_i.
\]
\item $\psi_{\clz, \clw}$ is a rational function whose poles all lie outside $\overline{\mathbb{B}^n}$.
\end{enumerate}
This function plays a significant role in a number of interpolation results developed in this paper, which are summarized in the following theorem:

\begin{thm}\label{thm intro interp}
Let $\clz = \{z_i\}_{i = 1}^m \subseteq \B^n$ and $\clw = \{w_i\}_{i = 1}^m \subseteq \D$ be interpolation data. Assume that $\clw \neq \{0\}$. Consider the function
\[
\psi_{\clz, \clw} = \sum_{j = 1}^m c_j \Sk(\cdot, z_j),
\]
as defined in \eqref{eqn: sect 0 psi}, where the scalars $\{c_1, \ldots, c_m\}$ are given by
\[
\begin{bmatrix}
c_1
\\
\vdots
\\
c_m
\end{bmatrix}
=
\begin{bmatrix}
\Sk(z_i, z_j)
\end{bmatrix}_{m \times m}^{-1}
\begin{bmatrix}
w_1
\\
\vdots\\
w_m
\end{bmatrix},
\]
and define
\[
\widetilde{\clm_{\clq_{\clz}}} = [\clq_{\clz}^{conj} \ominus \{ \overline{\psi_{\clz, \clw}}\}] \dot+[\clm(\Sn) \dot+ H^2_0(\Sn)].
\]
The following conditions are equivalent:
\begin{enumerate}
\item (Interpolation) $\clz$ and $\clw$ solve the interpolation problem.
\item (Qualitative property) $\Pi_{\clz, \clw} : (\clm_{\clq_{\clz}},\|\cdot\|_1) \raro \C$ is a contraction, where
\[
\Pi_{\clz, \clw} f = \int_{\Sn} \psi_{\clz, \clw} f d\sigma,
\]
for all $f \in \clm_{\clq_{\clz}}$, and
\[
\clm_{\clq_{\clz}} = \clq_{\clz}^{conj} \dot+ [\clm(\Sn) \dot+ H^2_0(\Sn)].
\]
\item (Quantitative property - I) $\text{dist}_{L^1(\Sn)} \Big(\frac{\overline{\psi_{\clz, \clw}}}{\|\psi_{\clz, \clw}\|_2^2}, \widetilde{\clm_{\clq_{\clz}}}\Big) \geq 1$.

\item (Quantitative property - II) $\text{dist}_{L^1(\Sn)} \left(\frac{1}{\sum_{i=1}^{m}|w_i|^2} \displaystyle \sum_{j=1}^m \overline{w_j} \overline{\Sk(\cdot, {z_j})}, \widetilde{\clm_{\clq_{\clz}}} \right) \geq 1$.

\item (Inner functions) There exists a sequence $\{\vp_k\}_k \subseteq \cli(\B^n)$ such that
\[
\lim_{k \raro \infty} \vp_k(z_j) = w_j,
\]
for all $j=1, \ldots, m$.
\end{enumerate}
\end{thm}

The distance formulas appearing in the quantitative parts of the above theorem, as well as in Theorem \ref{thm: all CLT}, is reminiscent of the classical Nehari theorem \cite{Nehari} (see the discussion following Corollary \ref{thm: NP interpolation distance}). Moreover, the equality case of the distance formulas has the following important consequence (see Theorem \ref{thm: extremal}):

\begin{thm}\label{thm intro: extreme}
Let $m > 1$ and let $\clz=\{z_i\}_{i=1}^m \subset \B^n$ and $\clw= \{w_i\}_{i=1}^m \subset \D$ be interpolation data. Consider the notation introduced in Theorem \ref{thm intro interp}. Then the following are equivalent:
\begin{enumerate}
\item $\clz$ and $\clw$ solve the extremal problem.
\item $\|\Pi_{\clz, \clw}\| = 1$.
\item $\text{dist}_{L^1(\Sn)}\Big(\frac{\overline{\psi_{\clz, \clw}}}{\|\psi_{\clz, \clw}\|_2^2}, \widetilde{\clm_{\clq_{\clz}}}\Big) = 1$.
\item $\text{dist}_{L^1(\Sn)} \left(\frac{1}{\sum_{i=1}^{m}|w_i|^2} \displaystyle \sum_{j=1}^m \overline{w_j} \overline{\Sk(\cdot, {z_j})}, \widetilde{\clm_{\clq_{\clz}}} \right) = 1$.
\end{enumerate}
\end{thm}

Our technique reveals a large class of extremal problems and functions. Indeed, given any nonconstant Schur function $\vp \in \cls(\B^n)$ and a set $\clz = \{z_i\}_{i=1}^m \subset \B^n$, one may conclude that $\clz = \{z_i\}_{i=1}^m$ and $\clw = \{w_i\}_{i=1}^m$ solve the interpolation problem, where
\[
w_i = \vp(z_i),
\]
for all $i=1, \ldots, m$. Therefore, constructing a solvable interpolation problem is relatively straightforward. The following result shows that a solvable extremal problem can likewise be constructed from a solvable interpolation problem. This construction relies on a normalization procedure and the commutant lifting theorem developed in this paper (see Theorem \ref{thm: int vs ext} for more details):

\begin{thm}\label{thm intro: int vs ext}
Let $m > 1$. Suppose $\clz = \{z_i\}_{i=1}^m \subset \B^n$ and $\clw = \{w_i\}_{i=1}^m \subset \D$ are interpolation data, with the $w_i$ not all equal. If $\clz$ and $\clw$ solve the interpolation problem, then $\clz$ and $\widehat{\clw}$ solve the extremal interpolation problem, where
\[
\widehat{\clw} = \{\widehat{w}_1, \ldots, \widehat{w}_m\} \subset \D,
\]
and
\[
\widehat{w}_j = \frac{1}{\|\Pi_{\clz, {\clw}}\|} w_j,
\]
for all $j=1, \ldots, m$, with $\Pi_{\clz, \clw}$ as in part (2) of Theorem \ref{thm intro interp}.
\end{thm}

We adopt the above normalization technique and apply our lifting theorem to obtain the following result, which directly connects solutions to the interpolation problem with solutions to the extremal problem (see Corollary \ref{cor: int vs ext}). 

\begin{thm}\label{thm intro: extrm int}
Let $m>1$. Let $\clz = \{z_i\}_{i=1}^m \subset \B^n$ and $\clw = \{w_i\}_{i=1}^m \subset \D$ be interpolation data, with the $w_i$ not all equal, and suppose that $\clz$ and $\clw$ solve the interpolation problem. Define
\[
\lambda_{\clz, \clw} = \inf \{\|\eta\|_\infty: \eta \in \cls(\B^n) \text{ interpolates } \clz \text{ and } \clw\}.
\]
Then there exists an extremal function $\vp \in \cle(\B^n)$ such that
\[
\lambda_{\clz, \clw} \vp,
\]
interpolates $\clz$ and $\clw$.
\end{thm}

In other words, $\lambda_{\clz,\clw}\vp$ is a minimum-norm interpolant for the data $\clz$ and $\clw$.

Kosi\'{n}ski and Zwonek \cite[Theorem 2]{K} provide a set of rational functions of degree at most two that yield a solution (up to biholomorphism) to each three-point extremal problem. As a result, whenever a three-point interpolation problem is solvable, Theorem \ref{thm intro: extrm int} yields an explicit interpolant obtained by scaling a function from the extremal solution set of Kosi\'{n}ski and Zwonek by a factor in $(0,1]$ (see Theorem \ref{thm: 3point rational} for more details).

We remark that the theory of Hilbert function spaces in several variables is of different levels of intricacy on $\B^n$ than on the polydisc $\D^n$. While solutions to the interpolation and commutant lifting problems in the polydisc setting have been developed in the recent paper \cite{KD}, as well as in several earlier works that established results of considerable depth and influence \cite{AM 1999, Ball}, the literature contains virtually no comparable attempts, let alone partial results, within the framework of Schur functions on $\B^n$ (however, see \cite{Amar, Cole, K}). This disparity also reflects the greater complexity of the function theory associated with Hilbert function spaces on $\B^n$ (see Section \ref{sec: remark} for more comments).

In particular, the connection between quotient modules of $H^2(\B^n)$ and inner functions developed here reveals a fundamental distinction between $\B^n$ and $\D^n$. Subsequently, a central theme of this paper is the study of the interplay between interpolation problems, inner and extremal functions, and their approximation properties on $\B^n$. This interplay reflects the inherent complexity of Hilbert function space theory on $\B^n$.

The remainder of the paper is organized as follows. Section \ref{sec: Carath approx} is devoted to the Pick and Carath\'{e}odory theorems on $\B^n$. Sections \ref{sec: pert}, \ref{sec: quality char}, \ref{sec: quant char}, and \ref{Sec: inner fn} present several characterizations of commutant lifting. Section \ref{sec: interpolation} addresses solutions to the interpolation problem, while Section \ref{sec: Carath interp} provides some solutions to the Carath\'{e}odory-Fej\'{e}r interpolation problem. Norm-preserving lifting is discussed in Section \ref{sec: norm pres lift}, and Section \ref{Sec: ext inter} is devoted to solutions to the extremal problem. Sections \ref{Sec: examples} and \ref{Sec: 3 point int} are mostly devoted to examples of lifting and interpolation. The paper concludes with a final remarks section in Section \ref{sec: remark}.

\newsection{Pick and Carath\'{e}odory theorems}\label{sec: Carath approx}

In this section, we establish two results that extend the classical theorems of Pick and Carath\'{e}odory from $\D$ to $\B^n$. We begin by recalling the extremal problem (see also Section \ref{sec: introd}).

Given interpolation data $\clz = \{z_i\}_{i=1}^m \subseteq \B^n$ and $\clw = \{w_i\}_{i=1}^m \subseteq \D$, we say that $\clz$ and $\clw$ solve the \textit{extremal problem} (or, more specifically, the \textit{$m$-point extremal problem}) if:
\begin{enumerate}
\item (Interpolation) $\clz$ and $\clw$ solve the interpolation problem.
\item (Extremal) If $\vp \in \cls(\B^n)$ interpolates $\clz$ and $\clw$, then
\[
\|\vp\|_\infty = 1.
\]
\end{enumerate}

For each natural number $m \in \mathbb{N}$, we denote by $\cle_m(\B^n)$ the set of all solutions to the $m$-point extremal problems:
\[
\cle_m(\B^n) = \{\varphi \in \cls(\B^n) : \varphi \text{ solves some } m\text{-point extremal problem}\}.
\]
In this context, it is important to observe that
\[
\cle_1(\B^n) = \emptyset.
\]
Define
\[
\cle(\B^n) = \bigcup_{m \in \mathbb{N}} \cle_m(\B^n).
\]
Recall that elements of $\cle(\B^n)$ are called \textit{extremal functions}. The classical one-variable case $n=1$ is of particular interest. In this case, we know that (recall Theorem \ref{thm E(D) = finite BP})
\[
\cle(\D) = \{\text{finite Blaschke products}\},
\]
while Pick's theorem yields another interpretation of the set $\cle(\D)$ (see Theorem \ref{thm: Intro Pick}): If the data $\clz = \{z_i\}_{i=1}^m \subseteq \D$ and $\clw = \{w_i\}_{i=1}^m \subseteq \D$ solve the interpolation problem, then there exists
\[
b \in \cle(\D),
\]
such that $b$ interpolates $\clz$ and $\clw$. As noted above, $b$ is a finite Blaschke product. In several variables, however, there is no satisfactory analogue of finite Blaschke products. We therefore replace finite Blaschke products by the class of extremal functions. Viewed in this light, the following theorem may be regarded as an extension of Pick's theorem from $\D$ to $\B^n$.

\begin{theorem}\label{thm: extr sol}
Let $\clz = \{z_i\}_{i=1}^m \subseteq \B^n$ and $\clw = \{w_i\}_{i=1}^m \subseteq \D$ be data. If $\clz$ and $\clw$ solve the interpolation problem, then there exists an extremal function
\[
\vp \in \cle(\B^n),
\]
that interpolates $\clz$ and $\clw$.
\end{theorem}
\begin{proof}
Without loss of generality, assume that the interpolation problem is not extremal. Define
\[
\cls_{\clz,\clw} = \{\varphi \in \cls(\B^n): \varphi \text{ interpolates } \clz \text{ and } \clw \}.
\]
Since $\clz$ and $\clw$ solve the interpolation problem, by our assumption,
\[
\cls_{\clz,\clw} \neq \emptyset.
\]
Pick
\[
z_{m+1} \in \B^n \setminus \{z_1, \dots, z_m\},
\]
and define
\[
M = \sup_{\varphi \in \cls_{\clz,\clw}} |\varphi(z_{m+1})|.
\]
Clearly, $M \leq 1$ (as $\cls_{\clz,\clw} \subseteq \cls(\B^n)$). We claim that
\[
M > 0.
\]
Since the interpolation problem in non-extremal, there exists $\varphi \in \cls_{\clz,\clw}$ such that
\[
\|\varphi\|_{\infty} < 1.
\]
If $\varphi(z_{m+1}) \neq 0$, then $M > 0$. Therefore, assume that
\[
\varphi(z_{m+1}) = 0.
\]
Choose a polynomial $p \in \C[z_1, \dots , z_n]$ such that
\[
\|p\|_{\infty} = 1,
\]
and
\[
p(z_i) = 0,
\]
for all $i=1, \ldots, m$, and
\[
p(z_{m+1}) \neq 0.
\]
Since $\|\varphi\|_{\infty} < 1$, there exists $t > 0$ such that
\[
\|\widetilde{\varphi}\|_{\infty} < 1,
\]
where
\[
\widetilde{\varphi} := \varphi + t p
\]
Clearly, $\widetilde{\varphi}$ interpolates $\clz$ and $\clw$; that is, $\widetilde{\varphi} \in \cls_{\clz,\clw}$. Also, $\widetilde{\varphi}(z_{m+1}) = t p(z_{m+1})$ implies that $\widetilde{\varphi}(z_{m+1}) \neq 0$. Then
\[
M \geq \big|\widetilde{\varphi}(z_{m+1})\big| > 0,
\]
proving the claim that $M > 0$. There exists a sequence $\{\varphi_r\}_r \subseteq \cls_{\clz,\clw}$ such that
\[
\lim_{r \raro \infty} |\varphi_r (z_{m+1})| = M.
\]
Since $\|\varphi_r\|_{\infty} \leq 1$, for $r \geq 1$, the sequence $\{\varphi_r\}_r$ is uniformly bounded. Hence, by Montel's theorem, there exists a subsequence $\{\varphi_{r_{k}}\}$ that converges uniformly on compact subsets of $\B^n$ to some function $\theta$. By Weierstrass' theorem, it follows that $\theta$ is analytic on $\B^n$. On the other hand, since $|\varphi_{r_{k}}(z)| \leq 1$ for all $k$, we have $|\theta(z)| \leq 1$ for all $z \in \B^n$, that is,
\[
\theta \in \cls(\B^n).
\]
Moreover, $\{\varphi_{r_k}\} \subseteq \cls_{\clz,\clw}$ implies
\[
\varphi_{r_{k}}(z_j) = w_j,
\]
for all $k \in \mathbb{N}$ and for all $j=1, \ldots, m$. Since $\varphi_{r_{k}}$ converges pointwise to $\theta$, it follows that
\[
\varphi_{r_{k}}(z_j)  \raro \theta(z_j),
\]
and therefore
\[
\theta(z_j) = w_j,
\]
for all $j=1, \ldots, m$. This proves that $\theta \in \cls_{\clz,\clw}$. Let
\[
w_{m+1} := \theta(z_{m+1}),
\]
and define data $\widetilde{\clz}$ and $\widetilde{\clw}$ by
\[
\widetilde{\clz} = \clz \cup \{z_{m+1}\},
\]
and
\[
\widetilde{\clw} = \clw \cup \{w_{m+1}\}.
\]
Clearly, $\theta \in \cls(\B^n)$ interpolates $\widetilde{\clz}$ and $\widetilde{\clw}$. We need to prove that $\widetilde{\clz}$ and $\widetilde{\clw}$ solve the extremal problem. Suppose $\eta \in \cls(\B^n)$ interpolates $\widetilde{\clz}$ and $\widetilde{\clw}$. We claim that $\|\eta\|_{\infty} = 1$. Suppose, to the contrary, that
\[
\|\eta\|_{\infty}<1.
\]
Once again, choose a polynomial $p\in \C[z_1,\ldots,z_n]$ such that
\[
\|p\|_{\infty}=1,
\]
and
\[
p(z_i)=0,
\]
for all $i=1, \ldots, m$, and
\[
p(z_{m+1})\neq 0.
\]
On the other hand,
\[
\varphi_{r_{k}}(z_{m+1}) \raro \theta(z_{m+1}) = w_{m+1},
\]
together with
\[
|\varphi_{r_{k}}(z_{m+1})| \raro M,
\]
implies
\[
|w_{m+1}| = |\theta(z_{m+1})| = M.
\]
Therefore,
\[
\left|\frac{w_{m+1}}{p(z_{m+1})}\right| > 0.
\]
Since $1 - \|\eta\|_{\infty}>0$, we may choose $\lambda>0$ so that
\[
0 < \lambda\left|\frac{w_{m+1}}{p(z_{m+1})}\right| < 1 - \|\eta\|_{\infty}.
\]
Set
\[
\lambda_0=\lambda \frac{w_{m+1}}{p(z_{m+1})},
\]
and define
\[
\psi=\eta+\lambda_0 p.
\]
For $i=1, \ldots, m$, we have $\psi(z_i) = w_i$, and also
\[
\|\psi\|_{\infty} \leq \|\eta\|_{\infty} + |\lambda_0| < \|\eta\|_{\infty} + 1 - \|\eta\|_{\infty} = 1.
\]
Thus $\psi \in \cls_{\clz,\clw}$. Moreover, we have
\[
\begin{split}
\big| \psi(z_{m+1}) \big| & = \big| \eta(z_{m+1}) + \lambda_0 p(z_{m+1}) \big|
\\
& = \big| w_{m+1} + \lambda_0 p(z_{m+1}) \big|
\\
& = \big| w_{m+1} + \lambda w_{m+1} \big|
\\
& = (1 + \lambda) \big| w_{m+1} \big|
\\
& > \big| w_{m+1} \big|
\\
& = |\theta(z_{m+1})|.
\end{split}
\]
Therefore, we have
\[
M = \sup_{\varphi \in \cls_{\clz,\clw}} |\varphi(z_{m+1})| = |\theta(z_{m+1})| < \big| \psi(z_{m+1}) \big|,
\]
which contradicts to the fact that $\psi \in \cls_{\clz,\clw}$. Therefore,
\[
\|\eta\|_{\infty} = 1.
\]
By construction, $\theta \in \cls_{\widetilde{\clz}, \widetilde{\clw}}$. Since the data $\widetilde{\clz}$ and $\widetilde{\clw}$ solve the extremal problem, it follows that $\theta$ is an extremal function. In particular, $\theta \in \cle_{m+1}(\B^n)$. Consequently, $\theta \in \cle(\B^n)$ and $\theta$ interpolates $\clz$ and $\clw$. This completes the proof of the theorem.
\end{proof}

The ``one-point extension'' technique in the proof bears some resemblance to an argument of Garnett \cite[p.~133, Corollary 1.9]{Garnett}. It is therefore worthwhile to compare the present theorem with Garnett's result. Such a comparison is especially intriguing because Garnett's argument is based on the theory of dual extremal problems, a line of investigation originating in the work of Havinson \cite{Havinson 1, Havinson 2}, Macintyre and Rogosinski \cite{Rog 1}, and Rogosinski and Shapiro \cite{Rog 2}.

We now prove a Carath\'{e}odory theorem on the ball. Recall that a sequence of functions $\{\vp_k\}_k \subseteq \cle(\B^n)$ is said to converge to a function $\vp$ in the weak*-topology of $L^{\infty}(\Sn)$ if
\[
\lim_{k \raro \infty} \int_{\Sn} \vp_k f d\sigma = \int_{\Sn} \vp f d\sigma,
\]
for all $f \in L^1(\Sn)$.

\begin{theorem}\label{thm:Car sv}
Let $\theta \in \cls(\B^n)$. Then there is a sequence $\{\vp_k\}_k \subseteq \cle(\B^n)$ such that
\[
\lim_{k \raro \infty} {\vp_k}= \theta,
\]
uniformly on compact subsets of $\B^n$. Moreover, this convergence also holds in the weak*-topology of $L^{\infty}(\Sn)$.
\end{theorem}
\begin{proof}
Fix a countable dense subset $\{z_j : j \in \N\}$ of $\B^n$. For each $j \in \N$, define
\[
w_j = \theta (z_j).
\]
We consider two cases.

\noindent \textsf{Case I:} In the first case, we assume that $\theta$ is not a unimodular constant. Since $\theta \in \cls(\B^n)$, we know that
\[
\{w_j: j \in \N\} \subseteq \D.
\]
For each $m \geq 1$, consider the interpolation data $\clz_m$ and $\clw_m$ defined by
\[
\clz_m = \{z_j\}_{j=1}^m,
\]
and
\[
\clw_m =  \{w_j\}_{j=1}^m.
\]
Evidently, $\clz_m$ and $\clw_m$ solve the interpolation problem, since $\theta$ itself is an interpolant function. By Theorem \ref{thm: extr sol} (specifically, see the proof of Theorem \ref{thm: extr sol}), for each $m \in \N$, there exists an extremal function $\varphi_m \in \cle_{m+1}(\B^n) \subseteq \cle(\B^n)$ such that $\varphi_m$ interpolates $\clz_m$ and $\clw_m$. Given that
\[
\|\varphi_m\|_{\infty} = 1,
\]
for $m \in \N$, we conclude that $\{\varphi_m\}_m$ is a uniformly bounded sequence of analytic functions. Therefore, by Montel's theorem, there is a subsequence $\{\varphi_{m_k}\}$ and a function $\eta: \B^n \raro \mathbb{C}$ such that
\[
\lim_{k \raro \infty} \vp_{m_k} = \eta,
\]
uniformly on compact subsets of $\B^n$. As noted in the proof of Theorem \ref{thm: extr sol}, we have $\eta \in \cls(\B^n)$, and
\[
\lim_{k \raro \infty} \varphi_{m_{k}}(z) = \eta(z),
\]
for all $z \in \B^n$.

\noindent Fix $j \in \N$. Since $\{m_k\}_k$ is a strictly increasing sequence of naturals, there exists a $k_0 \in \N$ such that
\[
m_k \geq j,
\]
for all $k \geq k_0$. We know that the function $\varphi_{m_{k}}$ interpolates the data $\clz_{m_k}$ and $\clw_{m_k}$. For each $k \geq k_0$, in particular, since $1 \leq j \leq m_k$, it follows that
\[
\varphi_{m_{k}}(z_j) = w_j.
\]
Since $w_j = \theta(z_j)$, we have
\[
\varphi_{m_{k}}(z_j) = \theta(z_j),
\]
for all $k \geq k_0$. Thus, passing to the limit (noting that $k_0$ depends only on $j$) we obtain
\[
\lim_{k \raro \infty} \varphi_{m_{k}}(z_j) = \theta(z_j).
\]
Therefore, we conclude that
\[
\theta(z_j) = \eta(z_j),
\]
for all $j \in \N$. Since $\{z_j : j \in \N\}$ is dense in $\B^n$, by continuity of $\eta$ and $\theta$, it follows that
\[
\eta = \theta.
\]
\noindent \textsf{Case II:} For the second, and remaining, case, we assume that $\theta$ is a unimodular constant. There exists $w \in \T$ such that
\[
\theta(z) = w,
\]
for all $z \in \B^n$. Choose a sequence $\{w_k\}_k \subset \D$ such that
\[
\lim_{k \raro \infty} w_k = w.
\]
From here, we proceed as in Case I, but the details are still necessary. For each $m \geq 1$, we consider the data $\clz_m$ and $\clw_m$ as follows:
\[
\clz_m = \{z_1, \ldots , z_m\} \subset \B^n,
\]
and
\[
\clw_m = \{w_m, \ldots ,w_m\} \subset \D.
\]
Clearly, $\clz_m$ and $\clw_m$ solve the interpolation problem. By Theorem $\ref{thm: extr sol}$, there exists $\varphi_m \in \cle(\B^n)$ interpolating $\clz_m$ and $\clw_m$. As before, since the sequence
\[
\{\varphi_m\}_m,
\]
is uniformly bounded, there exists a subsequence $\{\varphi_{m_k}\}$ that converges uniformly on compact sets to some $\eta \in \cls(\B^n)$. Again, as before, fix $j \geq 1$. There exists $k_0 \in \mathbb{N}$ such that
\[
m_k \geq j,
\]
for all $k \geq k_0$. We observe that $1 \leq j \leq m_k$ for all $k \geq k_0$. Now, $\varphi_{m_k}$ interpolates the data $\clz_{m_k}$ and $\clw_{m_k}$, so
\[
\varphi_{m_k}(z_j) = w_{m_k}.
\]
As $\lim_{k \raro \infty} w_k = w$, it follows that
\[
\lim_{k} \varphi_{m_k}(z_j) = w = \theta(z_j).
\]
Since $\{\varphi_{m_k}\}$ converges to $\eta$ uniformly on compact sets, we also have
\[
\lim_{k} \varphi_{m_k}(z_j) = \eta(z_j).
\]
Thus $\eta(z_j) = \theta(z_j)$. As the set $\{z_j : j \in \N\}$ is dense in $\B^n$, and both $\theta$ and $\eta$ are continuous, we conclude that
\[
\theta = \eta.
\]
Given that, in either case, $\theta = \eta$, the first part of the theorem follows by relabeling the subsequence $\varphi_{m_{k}}$ as $\vp_k$ for all $k \in \mathbb{N}$. Finally, given the proof of the first part, the second part of the result is now easy and follows exactly as in part (b) of \cite[Theorem 5.3]{Rud inner}. This completes the proof of the theorem.
\end{proof}

In view of the identification of $\cle(\D)$ with the set of all finite Blaschke products in the classical $n = 1$ case, as pointed out in Theorem \ref{thm E(D) = finite BP}, the above result is clearly a unit ball version of the Carath\'{e}odory theorem. This can also be considered as an extremal function version of the Aleksandrov-Rudin theorem (see Theorem \ref{thm: intro Alek Rud} or \cite[Theorem 5.3]{Rud inner}).

\newsection{Characterizations by perturbations}\label{sec: pert}

We now turn to quotient modules of $H^2(\B^n)$ and present a criterion for commutant lifting. We call a closed subspace $\clq \subseteq H^2(\B^n)$ a \textit{quotient module} if $T_{z_i}^* \clq \subseteq \clq$ for all $i=1, \ldots, n$. A closed subspace $\cls \subseteq H^2(\B^n)$ is called a \textit{submodule} if $\cls^\perp$ is a quotient module of $H^2(\B^n)$. Equivalently, $\cls$ is a joint invariant subspace under coordinate functions:
\[
z_i \cls \subseteq \cls,
\]
for all $i=1, \ldots, n$. We state a well-known result: Let $\cls$ be a closed subspace of $H^2(\B^n)$. Then $\cls$ is a submodule if and only if
\begin{equation}\label{eqn: phi submod}
\vp \cls \subseteq \cls,
\end{equation}
for all $\vp \in H^\infty(\B^n)$.

Let $\clq$ be a quotient module of $H^2(\B^n)$. In the context of Problem \ref{prob: CLT}, observe that if $X := S_\vp (= P_\clq T_\vp|_\clq)$ for some $\vp \in \cls(\B^n)$, then
\[
\|X\| = \|S_\vp\| = \|P_\clq T_\vp|_\clq\| \leq \|\vp\|_\infty \leq 1,
\]
that is, $X \in \clb_1(\clq)$. Moreover, it is evident that $X S_{z_i} = S_{z_i} X$ for all $i=1, \ldots, n$. Thus, $S_\vp$ is a contractive module map on $\clq$. Therefore, the commutant lifting problem is concerned with establishing the converse: whether a contractive module map necessarily arises as a compressed multiplication operator by a Schur class function.

In this section, we present a solution to this problem that is technically less involved, yet it will play an important role in the relatively more involved results that follow. At the same time, the solution provided here reveals a perturbation property that is of independent interest. To avoid trivialities in the lifting problem, we always assume that all quotient modules $\clq$ and module maps $X \in \clb(\clq)$ under consideration are nontrivial:
\begin{equation}\label{eqn: X & Q nonzero}
\clq \neq \{0\} \text{ and } X \neq 0.
\end{equation}

We introduce some notation. Given $n$ commuting bounded linear operators $\{T_i\}_{i=1}^n$ acting on some Hilbert space, and a multi-index $k = (k_1, \ldots, k_n) \in \Z_+^n$, we write
\[
T^k = T_1^{k_1} \cdots T_n^{k_n}.
\]
Similarly, we write $z^k = z_1^{k_1} \cdots z_n^{k_n}$.

Let $\clq$ be a quotient module of $H^2(\B^n)$; that is, $\clq$ is a closed subspace of $H^2(\B^n)$ satisfying $z_i \clq^\perp \subseteq \clq^\perp$ for all $i=1, \ldots, n$. We give a special attention to the function $P_\clq 1 \in \clq$. This function is of particular interest, as it is a joint cyclic vector in the sense of quotient modules:
\begin{equation}\label{eqn: Q cyclic}
\clq = \overline{\text{span}} \{S_z^k (P_\clq 1): k \in \Z^n_+\}.
\end{equation}
To see this, for each $k \in \Z^n_+$, we compute
\[
S_z^k P_\clq 1 = P_\clq T_z^k|_{\clq} P_\clq 1 = P_\clq T_z^k P_\clq 1 = P_\clq T_z^k 1 =  P_\clq z^k.
\]
The claim now follows from the fact that
\[
H^2(\B^n) = \overline{\text{span}} \{z^k: k \in \Z^n_+\}.
\]
It is well known that the structure of submodules and quotient modules of $H^2(\B^n)$ is complicated, and unlike the case $n=1$ (which is covered by Beurling’s theorem and the classical Sz.-Nagy–Foiaș model theory \cite{NF}), no satisfactory characterization or representation is available. This complexity presents a major obstacle in understanding the several-variable commutant lifting theorem as well as the interpolation problem (and vice versa as well). Our first characterization of commutant lifting relies on perturbations of certain natural functions and establishes a connection with submodules of $H^2(\B^n)$.

We present the lifting theorem for a more general framework, as it will be used in subsequent results. Specifically, we consider the algebra $\cla$ for lifting maps, where
\[
\cla = H^\infty(\B^n) \text{ or } A(\B^n),
\]
where $A(\B^n)$ denotes the ball algebra. Recall that $A(\mathbb{B}^n)$ is the Banach algebra consisting of all analytic functions on $\mathbb{B}^n$ that admit a continuous extension to $\overline{\B^n}$.

Note that the statement that a contractive module map $X$ acting on a quotient module admits a lift means, as usual, the existence of a Schur function $\vp \in \cls(\B^n)$ such that $X = S_\vp$. If $\vp$ lies merely in $\cla$ (and is not necessarily a Schur function), then we say that $X$ admits a lift to $\cla$ (in this case, $X$ is not necessarily contractive).

Let $X \in \clb(\clq)$ be a module map (recall that module maps are always assumed to be nonzero). Define
\[
\psi_X = X(P_\clq 1).
\]
In what follows, we will suppress the subscript and simply write $\psi$ (as $X$ will be clear from the context). This function
\[
\psi =  X(P_\clq 1) \in \clq,
\]
will play a central role throughout the developments that follow. Before we proceed, we point out that $\psi$ is never the zero function:
\[
\psi \neq 0.
\]
Indeed, assume that for a nonzero quotient module $\clq$ (as we always assume, see \eqref{eqn: X & Q nonzero}) of $H^2(\B^n)$, and for a module map $X \in \clb(\clq)$, suppose for contradiction that
\[
\psi = X (P_{\clq} 1) = 0.
\]
Since $X$ is a module map on $\clq$, it follows that
\[
X S_z^k = S_z^k X,
\]
and hence
\[
X (S_z^k P_{\clq} 1) = 0,
\]
$k \in \Z_+^n$. By \eqref{eqn: Q cyclic}
\[
\clq = \overline{\text{span}} \{S_z^k P_{\clq} 1: k \in \Z_+^n\},
\]
which immediately implies that $X = 0$, contradicting our assumption that $X \neq 0$ (again, see \eqref{eqn: X & Q nonzero}). Therefore, $\psi \neq 0$. This observation will also be important in the interpolation problem in Section \ref{sec: interpolation}. For now, we prove a lifting characterization in terms of perturbation.

\begin{theorem}\label{lemma: lift ball 1}
Let $\clq$ be a quotient module of $H^2(\B^n)$, and let $X \in \clb(\clq)$ be a module map. Then $X$ admits a lift to $\cla$ if and only if there exists $\vp \in \cla$ such that
\[
\psi - \vp \in \clq^{\perp}.
\]
Moreover, in this case $X = S_{\vp}$.
\end{theorem}
\begin{proof}
Suppose $X$ admits a lift. That is, there exists $\vp \in \cla$ such that
\[
X = S_{\vp}.
\]
Then
\[
\psi = X (P_{\clq} 1) = S_{\vp} (P_{\clq} 1) = P_{\clq} T_{\vp} P_{\clq} 1.
\]
Since $\clq ^\perp$ is a submodule, it follows that (see \eqref{eqn: phi submod})
\[
\vp \clq^{\perp} \subseteq \clq^{\perp},
\]
and hence $P_{\clq} T_{\vp} P_{\clq} = P_{\clq} T_{\vp}$. This implies
\[
\psi = P_{\clq} \vp,
\]
so we conclude that
\[
\psi- \vp = P_{\clq} \vp - \vp = - P_{\clq ^\perp} \vp.
\]
In other words, we have
\[
\psi- \vp  \in \clq^\perp.
\]
Conversely, assume that $\psi - \vp \in \clq^{\perp}$. Therefore, $P_\clq(\psi- \vp) = 0$. Since $\psi \in \clq$, we must have $P_\clq \psi = \psi$, and hence
\[
\psi = P_\clq \vp.
\]
Pick $k \in \Z_+^n$. As $X$ is a module map, $XS_{z_i} = S_{z_i} X$ for all $i=1, \ldots, n$, and hence
\[
X S_z^k = S_z^k X.
\]
Therefore
\[
X S_z^k P_{\clq} 1 = S_{z}^k X P_{\clq} 1 = S_z^k \psi.
\]
Moreover, since $\vp \in H^\infty(\B^n)$ is analytic, it follows that $T_\vp$ is an analytic Toeplitz operator. As a result,
\[
T_z^k T_\vp = T_\vp T_z^k,
\]
which implies $S_z^k S_\vp = S_\vp S_z^k$ for all $k \in \Z_+^n$. Therefore, for all $k \in \Z_+^n$, we have
\[
\begin{split}
X (S_z^k P_{\clq} 1) & = S_z^k \psi
\\
& = S_z^k P_\clq \vp
\\
& = S_z^k S_\vp P_\clq 1
\\
& = S_\vp (S_z^k P_\clq 1).
\end{split}
\]
In view of \eqref{eqn: Q cyclic}, it follows that $X = S_\vp$, which completes the proof of the theorem.
\end{proof}

We now address the commutant lifting problem in its common form. Note that no specific algebraic property of the algebra $\cla$ has been used in the proof of the above theorem. The following result may be proved in exactly the same manner as Theorem \ref{lemma: lift ball 1}; one simply replaces $\cla$ by $\cls(\B^n)$ throughout the argument:

\begin{corollary}\label{lift lemma}
Let $\clq$ be a quotient module of $H^2(\B^n)$, and let $X \in \clb(\clq)$ be a module map. Then $X$ admits a lift if and only if there exists $\vp \in \cls(\B^n)$ such that
\[
\psi - \vp \in \clq^{\perp}.
\]
Moreover, in this case $X = S_{\vp}$.
\end{corollary}

We now specialize to Schur functions in the ball algebra $A(\mathbb{B}^n)$. Define
\[
\mathcal{SB}(\mathbb{B}^n) = \mathcal{S}(\mathbb{B}^n) \cap A(\mathbb{B}^n).
\]
In the setting of the above theorem, if, in addition, the lifting symbol $\vp$ is also in $\mathcal{SB}(\mathbb{B}^n)$, then we say that $X$ admits a lift to $\mathcal{SB}(\mathbb{B}^n)$. Again, the following holds with a proof similar to that of Theorem \ref{lemma: lift ball 1}:

\begin{corollary}\label{lemma: lift ball algbera}
Let $\clq$ be a quotient module of $H^2(\B^n)$, and let $X \in \clb(\clq)$ be a module map. Then $X$ admits a lift to $\mathcal{SB}(\mathbb{B}^n)$ if and only if there exists $\vp \in \mathcal{SB}(\mathbb{B}^n)$ such that
\[
\psi - \vp \in \clq^{\perp}.
\]
Moreover, in this case $X$ lifts to $S_{\vp}$.
\end{corollary}

As already noted for $n=1$, every contractive module map admits a lifting. However, as we will see (cf. Section \ref{Sec: examples}), the situation changes significantly in higher dimensions: unlike the one-variable case, where every contractive module map lifts by Sarason’s theorem \cite{D Sarason}, there exist contractive module maps that may or may not admit liftings. Thus, the classification of liftings in several variables becomes a nontrivial and meaningful problem.

Before concluding this section, we note that the perturbation problem above is, perhaps coincidentally, connected with another perturbation phenomenon for inner functions studied by Rudin \cite{Rud inner}, a connection that we will discuss further in Section \ref{Sec: examples}.

\section{Qualitative characterizations}\label{sec: quality char}

In this section, we go deeper into the structure of $H^2(\B^n)$. All notations and preliminaries presented at the beginning of this section are standard for $H^2(\B^n)$. For further details, we refer the reader to the classic \cite{Rud} (also see \cite{Kor}).

Let $\zeta \in \Sn$ and let $\alpha >1$. In analogy with the notion of an angular region in one variable, define $\cld_{\alpha}(\zeta)$ by (see \cite[5.4.1]{Rud})
\[
\cld_{\alpha}(\zeta) = \Big\{z \in \B^n : |1 - \la z, \zeta \ra| < \frac{\alpha}{2} (1 - \|z\|^2)\Big\}.
\]
A function $f: \B^n \raro \C$  is said to have $K-$limit $c$ at $\zeta \in \Sn$ if, for every $\alpha > 1$,
\[
\lim_{\substack{z \raro \zeta\\ z \in \cld_{\alpha}(\zeta)}} f(z) = c.
\]
If this limit exists, we write (see \cite[5.4.6]{Rud})
\[
(K-\lim f)(\zeta) = c.
\]
The notion of $K$-limits gives boundary values of the Hardy functions as follows (see \cite[Theorem 5.6.6]{Rud}): For each $f \in H^2(\B^n)$, the function $f^*$ is defined almost everywhere on $\Sn$, where
\begin{equation}\label{eqn: f radial}
f^*(\zeta) = (K-\lim f)(\zeta),
\end{equation}
for $\zeta \in \Sn$ a.e. Define
\[
H^2(\Sn) = \overline{\{f|_{\Sn}: f \in A(\B^n)\}}^{L^2(\Sn)},
\]
and
\[
H^\infty(\Sn) = H^2(\Sn) \cap L^\infty(\Sn).
\]
By the corollary preceding Theorem 5.6.9 of \cite{Rud}, the assignment
\[
U f = f^*,
\]
for all $f \in H^2(\B^n)$, defines a unitary operator $U: H^2(\B^n) \rightarrow H^2(\Sn)$. In view of this, we will often identify $H^2(\B^n)$ with $H^2(\Sn)$ via the unitary $U$. Note that the set of polynomials
\[
\{p^*: p \in \C[z_1, \dots z_n]\},
\]
is dense in $H^2(\Sn)$. We also remark that if $K-\lim f$ exists for a function $f$, then
\[
K-\lim \bar{f} = \overline{K-\lim f}.
\]
The key tool in the above identification is the use of Poisson integral techniques, which rely on the Poisson kernel. The Poisson kernel on the unit ball is defined by \cite[Definition 3.3.1]{Rud}
\[
P(z, \zeta) = \frac{(1 - \|z\|^2)^n}{|1 - \langle z, \zeta \rangle|^{2n}},
\]
for all $z \in \B^n$ and $\zeta \in \Sn$. The Poisson integral $\clp[\vp]$ of a function $\vp \in L^1(\Sn)$ is defined by \cite[page 41]{Rud}
\[
\clp[\vp](z) = \int_{\Sn} P(z, \zeta) \vp(\zeta) d\sigma(\zeta),
\]
for all $z \in \B^n$. We claim that
\begin{equation}\label{eqn: Poisson P f}
f = \clp[f^*],
\end{equation}
for all $f \in H^2(\B^n)$. Indeed, since $f^* \in H^2(\Sn)$, \cite[Theorem 5.6.8(b)]{Rud} implies
\[
\clp[f^*] \in H^2(\B^n).
\]
Then
\[
g := \clp[f^*] \in H^2(\B^n),
\]
and (again, see \cite[Theorem 5.6.8(b)]{Rud})
\[
g^* = (\clp[f^*])^* = f^*,
\]
a.e. on $\Sn$. Thus, $g^*, f^* \in H^2(\Sn)$ and $g^* = f^*$. Since $Uh = h^*$ (recall the unitary $U: H^2(\B^n) \raro H^2(\Sn)$ above), it follows $g = f$, that is, $f = \clp[f^*]$. This completes the proof of the claim. We also have the following:
\[
\lim_{r\raro 1} \int_{\Sn} |f^* - f_r|^2 d\sigma = 0,
\]
where, $f_r$, $0< r < 1$, is defined by
\[
f_r(\zeta) = f(r \zeta),
\]
for $\zeta \in \Sn$ a.e. Note that $r \zeta = (r \zeta_1, \ldots, r \zeta_n)$. The above claim follows from \cite[Theorem 5.6.6]{Rud}.

Now we set up the notations needed for the lifting theorem. Although some of these have already appeared in the introduction, we present them again here with the necessary elaboration. If $S$ is a subset of $H^2(\Sn)$, then we define $S^{conj}$ by
\[
S^{conj} = \{\overline{f} : f \in S\}.
\]
The space of mixed functions is defined by
\[
\clm(\Sn) = L^2(\Sn) \ominus [H^2(\Sn) + H^2(\Sn)^{conj}].
\]
From the definition of $\clm(\Sn)$ it is clear that $\clm(\Sn)$ is self-adjoint, that is,
\[
\clm(\Sn) = \clm(\Sn)^{conj}.
\]
Moreover, we note that if $\cls$ is a closed subspace of $L^2(\Sn)$ then $\cls^{conj}$ is also a closed subspace of $L^2(\Sn)$. Recall that
\[
H^2_0(\Sn) = \{f^* \in H^2(\Sn): f \in H^2(\B^n), f(0) = 0\},
\]
equivalently
\[
H^2_0(\Sn) = \{f^* \in H^2(\Sn) : \langle f, 1 \rangle_{H^2(\B^n)} = 0\} =  H^2(\Sn) \ominus \{1\}.
\]
Fix a quotient module $\clq$ of $H^2(\Sn)$ and define
\[
\clm_{\clq} = \clq^{conj} \dot+ [\clm(\Sn) \dot+ H^2_0(\Sn)].
\]
where $\dot+$ denotes the skew sum in the Banach space $L^1(\Sn)$. Note that here we are using the identification of $H^2(\B^n)$ with $H^2(\Sn)$ via the $K$-limits of the Hardy functions (via the unitary operator $U$ defined as above). In other words, we consider $\clm_{\clq}$ as a subspace of $L^1(\Sn)$:
\[
(\clm_{\clq}, \|\cdot\|_1) \subset L^1(\Sn).
\]
Now if we consider $\clm_\clq$ to be a subspace of the Hilbert space $L^2(\Sn)$, then the skew sum mentioned above becomes an orthogonal sum. We keep this observation as a lemma for future reference:

\begin{lemma}\label{H^2 conj orth}
Let $\clq$ be a quotient module of $H^2(\B^n)$. Then $\clm(\Sn)$, $H^2_0(\Sn)$, and $\clq^{conj}$ are pairwise orthogonal closed subspaces of $L^2(\Sn)$. Moreover,
\[
H^2_0(\Sn) \perp H^2(\Sn)^{conj}.
\]
\end{lemma}

\begin{remark}\label{remark: H^2+H^2 conj closed}
Using the identification of $H^2(\B^n)$ and $H^2(\Sn)$, we have
\[
H^2(\Sn) + H^2(\Sn)^{conj} = H^2_0(\Sn) \oplus H^2(\Sn)^{conj},
\]
which readily implies that $H^2(\Sn) + H^2(\Sn)^{conj}$ is a closed subspace of $L^2(\Sn)$.
\end{remark}

We also recall the duality of the $L^p$-spaces in our present context. For each $\vp \in L^\infty(\Sn)$, define $\chi_\vp \in (L^1(\Sn))^*$ by
\[
\chi_\vp f = \int_{\Sn} f \vp \,d\sigma,
\]
for all $f \in L^1(\Sn)$. This yields the duality (a linear isometry and bijective map)
\begin{equation}\label{eqn: dual}
(L^1(\Sn))^* \cong L^\infty(\Sn).
\end{equation}
This duality will play a key role in what follows, and in particular, in the quantitative lifting theorem:

\begin{theorem}\label{contr lift}
Let $\clq \subseteq H^2(\B^n)$ be a quotient module and let $X \in \clb_1(\clq)$ be a module map. Define $X_{\clq} : \clm_{\clq} \raro \C$ by
\[
X_{\clq} f = \int_{\Sn} \psi f d\sigma,
\]
for all $f \in \clm_{\clq}$, where
\[
\psi = X (P_{\clq} 1).
\]
Then $X$ admits a lift if and only if $X_{\clq} : (\clm_{\clq}, \|.\|_{1}) \raro \C$ is a contraction.
\end{theorem}
\begin{proof}
First, we assume that $X \in \clb_1(\clq)$ admits a lift. That is, there exists $\vp \in \cls(\B^n)$ such that $X = S_\vp$. By Corollary \ref{lift lemma}, we know that
\[
\psi - \vp \in \clq^\perp.
\]
Since $\vp$ is a Schur function, by taking $K$-limit, we obtain
\[
\|\vp\|_{H^{\infty}(\Sn)} \leq 1.
\]
By the duality \eqref{eqn: dual}, we know that $\chi_{\vp} \in (L^1(\Sn))^*$, where
\[
\chi_{\vp} f = \int_{\Sn} \vp f \, d\sigma,
\]
for all $f \in L^1(\Sn)$. We also know (by the duality) that
\[
\|\chi_{\vp}\| = \|\vp\|_{H^{\infty}(\Sn)} \leq 1.
\]
We claim that
\begin{equation}\label{eqn: chi = vp}
\chi_\vp|_{\clm_\clq} = X_\clq.
\end{equation}
To prove this, first we pick $f \in \clq^{conj}$. We have
\[
\begin{split}
X_{\clq} f & = \int_{\Sn} \psi f d\sigma
\\
& = \int_{\Sn} \psi \overline{\overline{f}} d\sigma
\\
& = \langle \psi, \overline{f}\rangle_{H^2(\Sn)}
\\
& = \langle \psi, \overline{f} \rangle_{H^2(\B^n)}.
\end{split}
\]
Since $\psi - \vp \in \clq^{\perp}$, it follows that $\langle \psi - \vp, \overline{f}\rangle_{H^2(\B^n)} = 0$, and hence
\[
\langle\psi, \overline{f}\rangle_{H^2(\B^n)} = \langle\vp, \overline{f}\rangle_{H^2(\B^n)}.
\]
This implies
\[
\begin{split}
X_{\clq} f & = \langle\vp, \overline{f}\rangle_{H^2(\B^n)}
\\
& = \langle \vp, \overline{f}\rangle_{H^2(\Sn)}
\\
& = \int_{\Sn} \vp f d\sigma
\\
& = \chi_{\vp} f,
\end{split}
\]
that is,
\[
X_{\clq}|_{\clq^{conj}} = \chi_\vp|_{\clq^{conj}}.
\]
For $f \in \clm(\Sn)$, we have
\[
X_{\clq} f = \int_{\Sn} \psi f d\sigma = \langle f, \overline{\psi}\rangle_{L^2(\Sn)} = 0 = \chi_{\vp} f.
\]
This shows that
\[
X_{\clq}|_{\clm(\Sn)} = \chi_\vp|_{\clm(\Sn)}.
\]
Finally, if $f \in H^2_0(\Sn)$, then, as $\clq^{conj} \perp H^2_0(\Sn)$, we have
\[
X_{\clq} f = \int_{\Sn} \psi f d\sigma = 0,
\]
and on the other hand, as $\vp \in H^2(\B^n)$ and $f \in H^2_0(\Sn)$, we have
\[
\chi_{\vp} f = \int_{\Sn} \vp f d\sigma = 0 = X_{\clq} f.
\]
Therefore,
\[
X_{\clq}|_{H^2_0(\Sn)} = \chi_\vp|_{H^2_0(\Sn)} = 0,
\]
and proves the claim that $X_{\clq} = \chi_{\vp}|_{\clm_{\clq}}$. By using this, we have
\[
\|X_{\clq}\| \leq \|\chi_{\vp}\| = \|\vp\|_{\infty} \leq 1,
\]
which proves that $X_{\clq}$ is a contractive functional on $\clm_{\clq}$.

\noindent For the converse direction, assume that $X_{\clq} : (\clm_{\clq}, \|.\|_1) \longrightarrow \mathbb{C}$ is a contractive functional. By the Hahn-Banach extension theorem, there exists $\theta \in L^{\infty}(\Sn)$ such that
\[
\chi_\theta|_{\clm_\clq} = X_\clq,
\]
and
\[
\|\theta\|_\infty = \|X_\clq\| \leq 1.
\]
For $f \in \clq^{conj}$, as $\chi_{\theta} f = X_{\clq} f $, we have
\[
\int_{\Sn} \theta f d\sigma = \int_{\Sn} \psi f d\sigma,
\]
and hence
\be \label{conj perp}
\int_{\Sn} (\theta - \psi)f d\sigma = 0.
\ee
We claim that $\theta \in H^2(\Sn)$. To this end, we first note that $\theta \in L^2(\Sn)$. Let $f \in \clm(\Sn)$. As $\clm(\Sn)$ is self-adjoint, we have $\overline{f} \in \clm(\Sn)$, and hence
\[
\begin{split}
\langle \theta, f \rangle_{L^2(\Sn)} & = \int_{\Sn} \theta \overline{f} d\sigma
\\
& = \chi_{\theta} \overline{f}
\\
& = X_{\clq} \overline{f}
\\
& = \int_{\Sn} \psi \overline{f} \,d\sigma
\\
& = \langle \psi, f \rangle_{L^2(\Sn)}
\\
& = 0.
\end{split}
\]
This implies
\[
\theta \in L^2(\Sn) \ominus \clm(\Sn).
\]
In view of Remark $\ref{remark: H^2+H^2 conj closed}$, we know that the subspace $H^2(\Sn) + H^2(\Sn)^{conj}$ is closed in $L^2(\Sn)$, which yields
\[
\begin{split}
L^2(\Sn) \ominus \clm(\Sn) & = L^2(\Sn) \ominus \Big(L^2(\Sn) \ominus [H^2(\Sn) + H^2(\Sn)^{conj}]\Big)
\\
& = H^2(\Sn) + H^2(\Sn)^{conj},
\end{split}
\]
and consequently
\[
\theta \in  H^2(\Sn) + H^2(\Sn)^{conj}.
\]
There exists $\eta_1 \in H^2(\B^n)^{conj}$ and $\eta_2 \in H^2(\B^n)$ such that (recall that $\eta_j^*$ denotes the $K$-limit function of $\eta_j$, $j=1, 2$)
\[
\theta = {\eta_1}^* + {\eta_2}^*.
\]
Define
\[
\vp_1 = \eta_1 - \eta_1(0),
\]
and
\[
\vp_2 = \eta_2 + \eta_1(0).
\]
Then $\vp_1^* \in H^2_0(\Sn)^{conj}$ and $\vp_2^* \in H^2(\Sn)$, and
\[
\theta = \vp_1^* + \vp_2^*.
\]
As $\vp_1^* \in H^2_0(\Sn)^{conj}$, we have $\overline{\vp_1^*} \in H^2_0(\Sn)$, and hence
\[
\begin{split}
\langle \theta, \vp_1^* \rangle_{L^2(\Sn)} & = \int_{\Sn} \theta \overline{\vp_1^*} d\sigma
\\
& = \chi_{\theta} (\overline{\vp_1^*})
\\
& = X_{\clq} (\overline{\vp_1^*})
\\
& = \int_{\Sn} \psi^* \overline{\vp_1^*} d \sigma
\\
& = \langle \overline{\vp_1^*}, \overline{\psi^*}\rangle_{L^2(\Sn)}
\\
& = 0,
\end{split}
\]
which implies
\[
\langle \theta, \vp_1^* \rangle_{L^2(\Sn)} = 0.
\]
Now as $\theta = \vp_1^* + \vp_2^*$ and $\langle \vp_1^*, \vp_2^* \rangle_{L^2(\Sn)} = 0$, we have
\[
0 = \langle \theta, \vp_1^* \rangle_{L^2(\Sn)} = \langle \vp_1^* + \vp_2^*, \vp_1^* \rangle_{L^2(\Sn)} = \|\vp_1^*\|_{L^2(\Sn)} ^2,
\]
that is, $\vp_1^* = 0$, and hence
\[
\theta = \vp_2^* \in H^2(\Sn).
\]
We also have $\theta \in L^{\infty}(\Sn)$, so that
\[
\vp_2^* \in L^{\infty}(\Sn) \cap H^2(\Sn).
\]
Since $\vp_2 \in H^2(\B^n)$, \eqref{eqn: Poisson P f} implies that
\[
\vp_2 = \clp[\vp_2^*],
\]
where $\clp$ denotes the Poisson integral (see \cite[page 41]{Rud}). In other words, for any $z \in \B^n$, we have
\[
\vp_2(z) = \clp[\vp_2^*](z) = \int_{\Sn} P(z, \zeta) \vp_2^*(\zeta) \, d\sigma(\zeta),
\]
where
\[
P(z, \zeta) = \frac{(1 - \|z\|^2)^n}{|1 - \langle z, \zeta \rangle|^{2n}}.
\]
for all $z \in \B^n$ and $\zeta \in \Sn$, is the Poisson kernel. For $z \in \B^n$, we compute
\[
\begin{split}
|\vp_2(z)| & = \Big|\int_{\Sn} P(z, \zeta) \vp_2^*(\zeta) d\sigma(\zeta) \Big|
\\
& \leq \int_{\Sn} P(z, \zeta) |\vp_2^*(\zeta)| d\sigma(\zeta)
\\
& \leq \|\vp_2^*\|_{L^{\infty}(\Sn)},
\end{split}
\]
which proves that $\vp_2 \in H^{\infty}(\B^n)$, and
\[
\|\vp_2\|_{H^{\infty}(\B^n)} \leq \|\vp_2^*\|_{L^{\infty}(\Sn)} = \|\theta\|_{L^{\infty}(\Sn)} \leq 1.
\]
We recall from \eqref{conj perp} that
\[
\int_{\Sn} (\theta - \psi)f d\sigma = 0,
\]
for all $f \in \clq^{conj}$. Now $\theta = \vp_2^*$, for $\vp_2 \in H^{\infty}(\B^n)$ implies
\[
\vp_2 - \psi \in \clq^{\perp}.
\]
Finally, since $\|\vp_2\|_{\infty} \leq 1$, by Corollary \ref{lift lemma}, we conclude that $X$ admits a lift, namely $S_\theta$. In summary, we have proved the existence of a function $\theta \in \cls(\B^n)$ such that
\begin{equation}\label{eqn: |theta| = XQ 1}
\begin{cases}
\chi_\theta|_{\clm_\clq} = X_\clq,
\\
X = S_\theta, \text{ and }
\\
\|\theta\|_\infty = \|X_\clq\| \leq 1.
\end{cases}
\end{equation}
This completes the proof of the theorem.
\end{proof}

The kernel space of the functional $X_\clq$ will play an important role in obtaining a quantitative characterization of possible liftings (see Lemma \ref{dis norm} Theorem \ref{lift dis}).

Similar qualitative characterization of lifting also holds on the polydisc \cite{KD}. However, the proof of the above result differs from that in the polydisc case, as the approach used there is not applicable to the unit ball. Specifically, we avoided the use of an orthonormal basis for $L^2(\Sn)$, which was a key technique in the polydisc setting (see Section \ref{sec: remark} for further discussion). Moreover, we have the following remark:

\begin{remark}
In the proof of the sufficiency part of the theorem, although $X_\clq$ is assumed to be contractive, we do not use the contractivity of $X$ to obtain a lift. Instead, we show that the contractivity of $X_\clq$ forces $X$ to admit a lift, and hence $X$ is a contraction.
\end{remark}

Given a contractive module map $X \in \clb_1(\clq)$, it is natural to ask for the infimum of the norms of all symbols $\vp \in \cls(\B^n)$ satisfying
\[
X=S_\varphi.
\]
The following result shows that this quantity is precisely $\|X_\clq\|$. This not only highlights the power of the functional $X_{\clq}$ but also shows the accuracy and effectiveness of its construction.

\begin{theorem}\label{thm: norm XQ}
Let $\clq \subseteq H^2(\B^n)$ be a quotient module, and let $X \in \clb_1(\clq)$ be a module map that admits a lift. Then the functional $X_\clq: (\clm_\clq, \|\cdot\|_1) \raro \mathbb{C}$ from Theorem \ref{contr lift} satisfies
\[
\|X_\clq\| = \inf\{\|\vp\|_\infty: \vp \in \cls(\B^n) \text{ and } X = S_\vp\}.
\]
Moreover, there exists $\theta \in \cls(\B^n)$ such that $S_\theta$ lifts $X$ and $\|\theta\|_\infty = \|X_\clq\|$.
\end{theorem}
\begin{proof}
If $X = S_\vp$ for some $\vp \in \cls(\B^n)$, by the duality \eqref{eqn: dual}, we have
\[
\|\chi_{\vp}\| = \|\vp\|_{\infty} \leq 1.
\]
and by the necessary part of Theorem \ref{contr lift} (or see \eqref{eqn: chi = vp}), we know that
\[
\chi_\vp|_{\clm_\clq} = X_\clq.
\]
Therefore, we have
\[
\|X_\clq\| = \|\chi_\vp|_{\clm_\clq} \| \leq \|\chi_\vp\| = \|\vp\|_{\infty}.
\]
On the other hand, since $X$ admits a lift, $X_\clq$ is a contraction. From this point, we follow the proof of the sufficient part of Theorem \ref{contr lift} and arrive at identity \eqref{eqn: |theta| = XQ 1}: $X = S_\theta$ and $\chi_\theta|_{\clm_\clq} = X_\clq$ for some $\theta \in \cls(\B^n)$ with
\[
\|\theta\|_\infty = \|X_\clq\| \leq 1.
\]
This completes the proof of the theorem.
\end{proof}

In a sense, $\|X_\clq\|$ represents the optimal lifting norm for $X$. The above norm identity will play a crucial role in the subsequent developments of this paper.

\section{Quantitative characterizations}\label{sec: quant char}

We use all the characterizations obtained thus far in this paper to propose a quantitative characterization of lifting. Here, “quantitative’’ refers to a distance formula that bears a resemblance to the classical Nehari theorem \cite{Nehari}.

Given a normed linear space $V$, a nonempty set $S \subseteq V$ and a vector $x \in V$, the distance from $x$ to $S$ is defined by
\[
\text{dist}_V (x, S) = \inf \{\|x - y\|: y \in S\}.
\]
The following lemma is key to our quantitative analysis of lifting:

\begin{lemma}\label{dis norm}
Let $V$ be normed linear space, $\eta$ be a nonzero functional in $V^*$, and let $x \in V$. Then
\[
\text{dist}_V (x, \ker \eta) = \frac{|\eta(x)|}{\|\eta\|}.
\]
\end{lemma}
\begin{proof}
Let $y \in \ker \eta$. Then
\[
|\eta(x)| = |\eta(x - y)| \leq \|\eta\|\|x - y\|.
\]
and hence
\[
\frac{|\eta(x)|}{\|\eta\|} \leq \|x - y\|,
\]
so that
\[
\text{dist}_V (x, \ker \eta) \geq \frac{|\eta(x)|}{\|\eta\|}.
\]
For the reverse inequality, we recall that $\|\eta\| = \sup_{y \neq 0} \frac{|\eta(y)|}{\|y\|}$. Therefore, there exists a sequence $\{y_k\}_k \subseteq V$ such that
\[
\eta (y_k) \neq 0,
\]
for all $k$, and
\[
\frac{|\eta(y_k)|}{\|y_k\|} \raro \|\eta\|.
\]
Since
\[
x - \frac{\eta(x)}{\eta(y_k)} y_k \in \ker \eta,
\]
we have
\[
\begin{split}
\text{dist}_V (x, \ker \eta) & \leq \Big\| x - x - \frac{\eta(x)}{\eta(y_k)} y_k \Big\|
\\
& = \Big\|\frac{\eta(x)}{\eta(y_k)} y_k\Big\|
\\
& = \frac{|\eta(x)|}{|\eta(y_k)|}\|y_k\|
\\
& \longrightarrow \frac{|\eta(x)|}{\|\eta\|},
\end{split}
\]
and hence $\text{dist}_V (x, \ker \eta) \leq \frac{|\eta(x)|}{\|\eta\|}$.
\end{proof}

Given a quotient module $\clq \subseteq H^2(\B^n)$ and a module map $X \in \clb_1(\clq)$, we have defined
\[
\psi = X(P_\clq 1),
\]
and
\[
\clm_{\clq} = \clq^{conj} \dot+ [\clm(\Sn) \dot+ H^2_0(\Sn)].
\]
Under the given assumptions, also recall that $X_\clq : (\clm_\clq, \|\cdot\|_1) \raro \C$ is defined by
\[
X_\clq f = \int_{\Sn} \psi f d\sigma,
\]
for all $f \in \clm_\clq$ (cf. Theorem \ref{contr lift}). Define
\[
\widetilde{\clm_{\clq}} = [\clq^{conj} \ominus \{\overline{\psi}\}] \dot+ [\clm(\Sn) \dot+ H^2_0(\Sn)].
\]
This space is precisely the kernel of $X_\clq$:

\begin{lemma}\label{lemm: ker XQ}
Consider the functional $X_\clq: (\clm_\clq, \|\cdot\|_1) \raro \mathbb{C}$. Then $\ker X_\clq = \widetilde{\clm_{\clq}}$.
\end{lemma}
\begin{proof}
We write $X_\clq f = \int_{\Sn} \psi f d\sigma$ as
\[
X_{\clq} f = \la f, \overline{\psi} \ra_{L^2(\Sn)},
\]
for all $f \in \clm_\clq$, and consequently
\[
\ker X_{\clq} = [\clq^{conj} \ominus \{ \overline{\psi} \}] \dot+ [\clm(\Sn) \dot+ H^2_0(\Sn)],
\]
which completes the proof of the lemma.
\end{proof}

In view of the above and Lemma \ref{dis norm}, we have the following characterization of lifting. We again emphasize that, to avoid triviality, we always assume that the module maps $X \in \clb(\clq)$ are nonzero (see \eqref{eqn: X & Q nonzero}).

\begin{theorem}\label{lift dis}
Let $\clq \subseteq H^2(\B^n)$ be a quotient module, and let $X \in \clb_1(\clq)$ be a module map. Then $X$ admits a lift if and only if
\[
\text{dist}_{L^1(\Sn)}\Big(\frac{\overline \psi}{\|\psi\|_2^2}, \widetilde{\clm_{\clq}} \Big) \geq 1.
\]
\end{theorem}
\begin{proof}
Lemma \ref{lemm: ker XQ} yields
\[
\ker X_{\clq} = \widetilde{\clm_{\clq}}.
\]
Thus
\[
\text{dist}_{L^1(\Sn)}\Big(\frac{\overline \psi}{\|\psi\|_2^2}, \widetilde{\clm_{\clq}} \Big) = \text{dist}_{L^1(\Sn)}\Big(\frac{\overline \psi}{\|\psi\|_2^2}, \ker X_{\clq} \Big) .
\]
Observe that
\[
X_{\clq} \Big( \frac{\overline{\psi}}{\|\psi\|_2^2}\Big) = \int_{\Sn} \psi \frac{\overline{\psi}}{\|\psi\|_2^2} d\sigma = 1,
\]
We then have, by Lemma $\ref{dis norm}$, that
\[
\text{dist}_{L^1(\Sn)}\Big(\frac{\overline \psi}{\|\psi\|_2^2}, \widetilde{\clm_{\clq}}\Big) = \frac{1}{\|X_{\clq}\|}.
\]
Again, Theorem \ref{contr lift} states that $X$ admits a lift if and only if
\[
\|X_{\clq}\| \leq 1.
\]
In view of the above, this condition is equivalent to requiring that
\[
\text{dist}_{L^1(\Sn)}\Big(\frac{\overline \psi}{\|\psi\|_2^2}, \widetilde{\clm_{\clq}} \Big) \geq 1.
\]
This completes the proof of the theorem.
\end{proof}

In the above proof, we have highlighted the key identity (see Lemma \ref{lemm: ker XQ})
\[
\ker X_{\clq} = [\clq^{conj} \ominus \{ \overline{\psi} \}] \dot+ [\clm(\Sn) \dot+ H^2_0(\Sn)].
\]
This, in particular, once again provides deeper insight into the effectiveness of the functional $X_{\clq}$.

\newsection{Characterizations by approximations}\label{Sec: inner fn}

In this section, we establish a connection between the lifting problem and the classes of inner functions and extremal solutions. Recall from Section \ref{sec: Carath approx} that
\[
\cle(\B^n) = \{\vp \in \cls(\B^n): \vp \text{ is a solution to some extremal problem}\}.
\]
Let $\vp \in \cls(\B^n)$. Then there is a sequence of inner functions $\{\vp_k\}_k \subseteq \cli(\B^n)$ such that (see Theorem \ref{thm: intro Alek Rud})
\[
\lim_{k \raro \infty} \vp_k = \vp,
\]
uniformly on compact subsets of $\B^n$, and
\[
{w^*-}\lim_{k \raro \infty} \vp_k = \vp \text{ on } L^1(\Sn).
\]
Recall that ${w^*-}\lim_{k \raro \infty} \vp_k = \vp$ on a set $S \subseteq L^1(\Sn)$ means
\[
\lim_{k \raro \infty} \int_{\Sn} \vp_k f d\sigma = \int_{\Sn} \vp f d\sigma,
\]
for all $f \in S$. In Theorem \ref{thm:Car sv}, we proved the same result for the class of extremal functions $\cle(\B^n)$. Motivated by the roles of these two classes and their approximation properties for Schur functions, we shall use the notation
\[
\cla(\B^n),
\]
to denote either $\cli(\B^n)$ or $\cle(\B^n)$. We refer to $\cla(\B^n)$ as the class of \textit{approximation functions}. Combining the Aleksandrov-Rudin theorem (Theorem \ref{thm: intro Alek Rud}) with Theorem \ref{thm:Car sv}, we obtain the following result.

\begin{theorem}\label{thm: inner ext fns}
Let $\vp \in \cls(\B^n)$. Then there is a sequence $\{\vp_k\}_k \subseteq \cla(\B^n)$ such that
\[
\lim_{k \raro \infty} \vp_k = \vp,
\]
uniformly on compact subsets of $\B^n$. Moreover,
\[
{w^*-}\lim_{k \raro \infty} \vp_k = \vp \text{ on } L^1(\Sn).
\]
\end{theorem}

In other words, the last assertion of the above theorem asserts that
\[
\overline{\cla(\B^n)}^{w^*} = \cls(\B^n).
\]

With this, we continue in the setting of Theorem \ref{contr lift} and present yet another characterization of lifting in terms of inner functions. Given a quotient module $\clq \subseteq H^2(\B^n)$ and a module map $X \in \clb_1(\clq)$, we recall that (see Theorem \ref{contr lift})
\[
X_{\clq} f = \int_{\Sn} \psi f d\sigma,
\]
for $f \in \clm_{\clq}$, defines a functional $X_{\clq}: (\clm_\clq, \|\cdot\|_1) \longrightarrow \mathbb{C}$, where $\psi = X(P_{\clq} 1)$ and
\[
\clm_{\clq} = \clq^{conj} \dot+ [\clm(\Sn) \dot+ H^2_0(\Sn)].
\]
We connect the function $\psi$ with $\cla(\B^n)$ as follows:

\begin{theorem}\label{thm: lifting in terms of inner}
Let $\clq \subseteq H^2(\B^n)$ be a quotient module, and $X \in \clb_1(\clq)$ be a module map, and let $\psi = X(P_{\clq} 1)$. Then $X$ admits a lift if and only if there is a sequence $\{\vp_k\}_k \subseteq \cla(\B^n)$ such that
\[
{w^*-}\lim_{k \raro \infty} \vp_k = \psi \text{ on }\clm_\clq.
\]
\end{theorem}
\begin{proof}
Suppose $X$ admits a lift. From the proof of Theorem \ref{contr lift}, and specifically by equation \eqref{eqn: chi = vp}, there exists $\theta \in \cls(\B^n)$ such that for $f \in \clm_{\clq}$, we have
\[
X_{\clq} f = \int_{\Sn} \theta f  d\sigma,
\]
for all $f \in \clm_{\clq}$. By Theorem \ref{thm: inner ext fns}, there exists a sequence $\{\vp_k\}_k \subseteq \cla(\B^n)$ such that
\[
{w^*-}\lim_{k \raro \infty} \vp_k = \theta \text{ on } L^1(\Sn).
\]
In particular, $w^*-\lim_{k \raro \infty} \vp_k = \theta$ on $\clm_\clq$. Moreover, for each $f \in \clm_{\clq}$, since
\[
X_\clq f = \int_{\Sn} \psi f d\sigma = \int_{\Sn} \theta f d\sigma,
\]
it follows that
\[
\lim_{k \raro \infty} \int_{\Sn } \vp_k f d\sigma = \int_{\Sn} \psi f d\sigma,
\]
for all $f \in \clm_{\clq}$, that is, ${w^*-}\lim_{k \raro \infty} \vp_k = \psi$ on $\clm_\clq$.

\noindent Conversely, assume that ${w^*-}\lim_{k \raro \infty} \vp_k = \psi$ on $\clm_\clq$. For $f \in \clm_{\clq}$, we have
\[
\begin{split}
|X_{\clq} f| & = \Big|\int_{\Sn} \psi f d\sigma\Big|
\\
& = \Big|\lim_{k \raro \infty} \int_{\Sn} \vp_k f d\sigma\Big|
\\
& \leq \lim_{k \raro \infty} \int_{\Sn} |\vp_k f| d\sigma
\\
& \leq \|f\|_1,
\end{split}
\]
that is, $X_{\clq}$ on $(\clm_{\clq}, \|\cdot\|_1)$ is a contraction, and therefore, by Theorem \ref{contr lift}, it follows that $X$ admits a lift.
\end{proof}

Now, we examine the structure of $\clm_{\clq}$ more closely in order to prove a more economical version of the above theorem:

\begin{corollary}\label{Cor: lift inner q}
Let $\clq \subseteq H^2(\B^n)$ be a quotient module, and $X \in \clb_1(\clq)$ be a module map, and let $\psi = X(P_{\clq} 1)$. Then $X$ admits a lift if and only if there is a sequence $\{\vp_k\}_k \subseteq \cla(\B^n)$ such that
\[
w^*-\lim_{k \raro \infty} \vp_k = \psi \text{ on } \clq^{conj}.
\]
\end{corollary}
\begin{proof}
The necessary part is a special case of Theorem \ref{thm: lifting in terms of inner} (note that $\clq^{conj} \subseteq \clm_\clq$). For sufficiency, suppose there exists a sequence of functions $\{\vp_k\}_k \subseteq \cla(\B^n)$ such that $\lim_{k \raro \infty} \vp_k = \psi$ on $\clq^{conj}$. In view of Theorem \ref{thm: lifting in terms of inner}, it is enough to prove that
\[
w^*-\lim_{k \raro \infty} \vp_k = \psi \text{ on } \clm_\clq.
\]
To prove this, pick any $g \in \clm_{\clq} = \clq^{conj} \dot+ [\clm(\Sn) \dot+ H^2_0(\Sn)]$. Then, there exist $g_1 \in \clq^{conj}$, $g_2 \in \clm(\Sn)$, and $g_3 \in H^2_0(\Sn)$ such that
\[
g = g_1 + g_2 + g_3.
\]
Fix a $k$. Since $g_2 \in \clm(\Sn)$, we have
\[
\int_{\Sn}\vp_k g_2\, d\sigma = \la \vp_k, \overline{g_2} \ra_{L^2(\Sn)} = 0 = \la \psi, \overline{g_2} \ra = \int_{\Sn} \psi g_2 d\sigma,
\]
and also
\[
\int_{\Sn}\vp_k g_3 d\sigma = \langle \vp_k g_3, 1 \rangle = \vp_k(0) g_3 (0) = 0,
\]
as $g_3 \in H^2_0(\Sn)$. Also, since
\[
\int_{\Sn} \psi g_3 d\sigma = \langle g_3, \overline{\psi} \rangle = 0,
\]
we have
\[
\int_{\Sn} \vp_k g d\sigma = \int_{\Sn} \vp_k (g_1 + g_2 + g_3) d\sigma = \int_{\Sn} \vp_k g_1 d\sigma.
\]
In view of our assumption, as $g_1 \in \clq^{conj}$, we have
\[
\int_{\Sn} \vp_k g_1 d\sigma \raro \int_{\Sn} \psi g_1 d\sigma,
\]
and consequently
\[
\int_{\Sn} \vp_k g d\sigma = \int_{\Sn} \vp_k g_1 d\sigma \raro \int_{\Sn} \psi g_1 d\sigma = \int_{\Sn} \psi g d\sigma.
\]
This completes the proof of the corollary.
\end{proof}

We now focus on lifting of module maps on finite-dimensional quotient modules. We specialize to a weak topology on $L^\infty(\Sn)$. Let $S \subseteq L^1(\Sn)$. For each $f \in S$, define $\hat{f} \in (L^{\infty}(\Sn))^*$ by
\[
\hat{f}(\vp) = \int_{\Sn} \varphi f d\sigma,
\]
for all $\vp \in L^\infty(\Sn)$, and set
\[
\widehat{S} := \{\hat{f}: f \in S\} \subseteq (L^{\infty}(\Sn))^*.
\]
We let $\tau_S$ denote the weak-topology on $L^{\infty}(\Sn)$ generated by the family of functionals $\widehat{S}$. In other words,
\[
(L^{\infty}(\Sn), \tau_S),
\]
is the smallest topology on $L^{\infty}(\Sn)$ for which each functional $\hat{f} : L^{\infty}(\Sn) \raro \C$, $f \in S$, is continuous. Therefore, we have the following: Let $\{\vp_i\}$ be a net in $L^{\infty}(\Sn)$. Then
\[
\{\vp_{i}\} \longrightarrow \vp \text{ in } (L^{\infty}(\Sn), \tau_{\cls}),
\]
if and only if
\[
\hat{f}(\vp_i) \longrightarrow \hat{f}(\vp),
\]
for all $f \in S$. The following result establishes a connection between the weak topology introduced above and the lifting on finite-dimensional quotient modules:

\begin{theorem}\label{thm: lift fd}
Let $\clq \subseteq H^2(\B^n)$ be a finite-dimensional quotient module, and $X \in \clb_1(\clq)$ be a module map. Then $X$ admits a lift if and only if there is a sequence of functions $\{\vp_k\}_k \subseteq \cls(\B^n)$ such that
\[
\vp_k \longrightarrow \psi \text{ in } (L^{\infty}(\Sn), \tau_{\clq^{conj}}),
\]
where
\[
\psi = X(P_{\clq}1).
\]
\end{theorem}
\begin{proof}
Since $\clq$ is finite-dimensional, it follows that (see \cite[Chapter 2, Remark 2.5.5]{CG})
\[
\clq \subseteq H^\infty(\B^n).
\]
In particular, we have $\psi \in L^\infty(\Sn)$. Now, convergence
\[
\vp_k \longrightarrow \psi \text{ in } (L^{\infty}(\Sn), \tau_{\clq^{conj}}),
\]
is equivalent to the convergence of
\[
\hat{f}(\vp_k) \longrightarrow \hat{f}(\psi),
\]
for all $f \in \clq^{cconj}$. Equivalently, we have
\[
\int_{\Sn} \vp_k f d\sigma \longrightarrow \int_{\Sn} \psi f d\sigma,
\]
for all $f \in \clq^{conj}$. The result now follows from Corollary \ref{Cor: lift inner q}.
\end{proof}

Consider the set of all inner functions as a subset of $L^{\infty}(\Sn)$. Let $\clq$ be a finite-dimensional quotient module of $H^2(\B^n)$ and let $X \in \clb_1(\clq)$. Then the above result can be stated as follows: $X$ admits a lift if and only if
\[
\psi = X (P_\clq 1) \in \overline{\cla(\B^n)}^{(L^{\infty}(\Sn),\tau_{\clq^{conj}})},
\]
where $\cla(\B^n)$ stands for either the class $\cli(\B^n)$ of inner functions or the class $\cle(\B^n)$ of extremal solutions.

\newsection{Interpolation}\label{sec: interpolation}

When $n=1$, Sarason \cite{D Sarason} showed that the solution to the classical interpolation problem follows from his commutant lifting theorem. In a similar spirit, we apply the characterization of lifting to solve the interpolation problem for Schur functions on $\B^n$. In addition, we apply approximation results to give analogous characterizations of solvability of the interpolation problem.

Recall that $H^2(\B^n)$ is a reproducing kernel Hilbert space corresponding to the Szeg\"{o} kernel on $\B^n$ defined as
\[
K_n (z, w) = \frac{1}{(1 - \la z, w \ra)^n},
\]
for all $z, w \in \B^n$. Given $m$ distinct points $\clz = \{z_i\}_{i = 1}^m \subseteq \B^n$, define
\[
\clq_{\clz} = \text{span }\{K_n(\cdot , z_i): i=1, \ldots, m\}.
\]
Clearly, $\clq_\clz$ is an $m$-dimensional subspace of $H^2(\B^n)$. Since
\begin{equation}\label{eqn: T_z*i Kn}
T_{z_j}^* K_n(\cdot , v) = \overline{v_j} K_n(\cdot, v),
\end{equation}
for all $j=1, \ldots, n$, and $v =(v_1, \ldots, v_n) \in \B^n$, it follows that $\clq_{\clz}$ is also a quotient module of $H^2(\B^n)$. Observe that
\[
\clq_\clz \subset A(\B^n) \subset H^{\infty}(\B^n).
\]
Fix interpolation data $\clz = \{z_i\}_{i=1}^m \subseteq \B^n$ and $\clw = \{w_i\}_{i=1}^m \subseteq \D$. Following Sarason, we define $X_{\clz, \clw} \in \clb(\clq_{\clz})$ by
\begin{equation}\label{eqn: X* zw}
X_{\clz, \clw}^* \Sk(\cdot, z_j) = \overline{w_j} \Sk(\cdot, z_j),
\end{equation}
for all $j=1, \ldots, m$. By \eqref{eqn: T_z*i Kn}, it follows that $X_{\clz, \clw}$ is a module map. Following the construction of $\psi$ in Theorem \ref{contr lift}, define
\[
\psi = X_{\clz, \clw} (P_{\clq_{\clz}} 1).
\]
We claim that the $m \times m$ matrix
\[
\begin{bmatrix}
\Sk(z_i, z_j)
\end{bmatrix}_{m \times m} =
\begin{bmatrix}
\Sk(z_1, z_1) & \dots & \Sk(z_1, z_m)
\\
\vdots & \ddots & \vdots
\\
\Sk(z_m, z_1) & \dots & \Sk(z_m, z_m)
\end{bmatrix},
\]
is invertible. Indeed, we know that
\[
\Sk(z_i, z_j) = \langle \Sk(\cdot, z_j), \Sk(\cdot, z_i) \rangle,
\]
for all $i, j = 1, \ldots, m$, and that the set $\{\Sk(\cdot, z_i)\}_{i=1}^m$  consists of linearly independent vectors. Since the Szeg\"{o} kernel is positive definite and the points $z_1, \ldots, z_m$ are distinct, the above matrix is an invertible Gram matrix. Define scalars $c_1, \ldots, c_m$ by
\[
\begin{bmatrix}
c_1
\\
\vdots
\\
c_m
\end{bmatrix}
=
\begin{bmatrix}
\Sk(z_i, z_j)
\end{bmatrix}_{m \times m}^{-1}
\begin{bmatrix}
w_1
\\
\vdots\\
w_m
\end{bmatrix},
\]
and define $\psi_{\clz, \clw} \in \clq_\clz$ by
\begin{equation}\label{eqn: psi Z W}
\psi_{\clz, \clw} = \sum_{j = 1}^m c_j \Sk(\cdot, z_j).
\end{equation}
Note the important fact that
\begin{equation}\label{eqn: basic int fn}
\psi_{\clz, \clw}(z_j) = w_j,
\end{equation}
for all $j=1, \ldots, m$.

In the discussion preceding Theorem \ref{lemma: lift ball 1}, we have already observed that
\[
\psi = X (P_{\clq} 1) \neq 0,
\]
where $\clq$ is a quotient module of $H^2(\B^n)$ and $X \in \clb(\clq)$ is a module map. We reiterate that throughout this discussion we are working under the standing assumption that $\clq \neq \{0\}$ and $X \neq 0$ (see \eqref{eqn: X & Q nonzero}). Before relating the above $\psi$ to the function $\psi_{\clz, \clw}$ introduced in \eqref{eqn: psi Z W}, we first investigate the case when $\psi_{\clz, \clw} = 0$. Suppose
\[
\psi_{\clz, \clw} = 0.
\]
Then
\[
\sum_{j = 1}^m c_j K_n(\cdot, z_j) = 0.
\]
Since the set $\{K_n(\cdot, z_j): j = 1, \ldots, m\}$ is linearly independent, we have $c_j = 0$ for $j = 1, \ldots, m$. By using
\[
\begin{bmatrix}
c_1
\\
\vdots
\\
c_m
\end{bmatrix}
=
\begin{bmatrix}
\Sk(z_i, z_j)
\end{bmatrix}_{m \times m}^{-1}
\begin{bmatrix}
w_1
\\
\vdots\\
w_m
\end{bmatrix},
\]
we conclude that
\[
\begin{bmatrix}
w_1
\\
\vdots
\\
w_m
\end{bmatrix}
= 0,
\]
that is, $\clw = \{0\}$. This proves the following:
\[
\psi_{\clz, \clw} = 0 \Longleftrightarrow \clw = \{0\}.
\]
On the other hand, if $\clw = \{0\}$, then the interpolation problem corresponding to the data $\clz = \{z_1, \dots ,z_m\} \subset \B^n$ and $\clw$ is always solvable, with the zero function serving as a solution. Therefore, without loss of generality, we shall henceforth assume that the interpolation data $\clz$ and $\clw$ satisfy
\[
\psi_{\clz, \clw} \neq 0,
\]
equivalently, that
\[
\clw \neq \{0\}.
\]

So far, we have constructed two different kinds of functions, namely $\psi$ and $\psi_{\clz, \clw}$, from two distinct perspectives. However, they are the same:

\begin{proposition}\label{prop: basic Z and W}
Let $\clz = \{z_i\}_{i=1}^m \subseteq \B^n$ and $\clw = \{w_i\}_{i=1}^m \subseteq \D$ be interpolation data. Consider the module map $X_{\clz, \clw}$ as defined by \eqref{eqn: X* zw}. Then
\[
\psi = \psi_{\clz, \clw},
\]
where $\psi = X_{\clz, \clw} (P_{\clq_{\clz}} 1)$ and $\psi_{\clz, \clw}$ is the function defined in \eqref{eqn: psi Z W}.
\end{proposition}
\begin{proof}
Note that since $\psi \in \clq_\clz$, there exist scalars $\alpha_1, \dots \alpha_m$ such that
\[
\psi = \sum_{i = 1}^m \alpha_i \Sk(\cdot, z_i).
\]
For each $j=1, \ldots, m$, we have
\[
\begin{split}
\psi (z_j) & = \la \psi, \Sk(\cdot, z_j) \ra
\\
& = \la X_{\clz, \clw} (P_{\clq_{\clz}} 1), \Sk(\cdot, z_j) \ra
\\
& = \la P_{\clq_{\clz}} 1, X_{\clz, \clw}^* \Sk(\cdot, z_j) \ra
\\
& = \la P_{\clq_{\clz}} 1, \overline{w_j} \Sk(\cdot, z_j) \ra
\\
& = w_j \la 1,P_{\clq_{\clz}} \Sk(\cdot, z_j) \ra
\\
& = w_j \la 1, \Sk(\cdot, z_j) \ra
\\
& = w_j \Sk(z_j, 0)
\\
& = w_j,
\end{split}
\]
that is, $w_j = \psi(z_j) $, which implies
\[
w_j = \sum_{i = 1}^m \alpha_i \Sk(z_j, z_i).
\]
In other words, we have
\[
\begin{bmatrix}
\alpha_1
\\
\vdots
\\
\alpha_m
\end{bmatrix}
=
\begin{bmatrix}
\Sk(z_i, z_j)
\end{bmatrix}_{m \times m}^{-1}
\begin{bmatrix}
w_1
\\
\vdots\\
w_m
\end{bmatrix},
\]
and consequently, by comparing with the identity preceding \eqref{eqn: basic int fn}, we conclude that $c_j = \alpha_j$. Therefore, we obtain the desired identity: $\psi_{\clz, \clw} = \psi$.
\end{proof}

We are now in a position to solve the interpolation problem. The arguments developed in the proof will also be used in the results that follow.

\begin{theorem}\label{thm: NP interpolation contr}
Let $\clz = \{z_i\}_{i = 1}^m \subseteq \B^n$ and $\clw = \{w_i\}_{i = 1}^m \subseteq \D$ be interpolation data. Define
\[
\clm_{\clq_{\clz}} = \clq_{\clz}^{conj} \dot+ [\clm(\Sn) \dot+ H^2_0(\Sn)],
\]
and
\[
\psi_{\clz, \clw} = \sum_{j = 1}^m c_j \Sk(\cdot, z_j),
\]
where the scalars $\{c_1, \ldots, c_m\}$ are given by
\[
\begin{bmatrix}
c_1
\\
\vdots
\\
c_m
\end{bmatrix}
=
\begin{bmatrix}
\Sk(z_i, z_j)
\end{bmatrix}_{m \times m}^{-1}
\begin{bmatrix}
w_1
\\
\vdots\\
w_m
\end{bmatrix}.
\]
Then $\clz$ and $\clw$ solve the interpolation problem if and only if
\[
\Pi_{\clz, \clw} f = \int_{\Sn} \psi_{\clz, \clw} f d\sigma \qquad (f \in \clm_{\clq_{\clz}}),
\]
defines a contraction $\Pi_{\clz, \clw} : (\clm_{\clq_{\clz}},\|\cdot\|_1) \raro \C$.
\end{theorem}
\begin{proof}
Given the data $\clz$ and $\clw$, recall the construction of the module map $X_{\clz, \clw}$ on $\clq_\clz$ as in \eqref{eqn: X* zw}:
\[
X_{\clz, \clw}^* \Sk(\cdot, z_i) = \overline{w_i} \Sk(\cdot, z_i),
\]
for all $i=1, \ldots, m$. We claim that a Schur function $\vp \in \cls(\B^n)$ interpolates $\clz$ and $\clw$ if and only if
\begin{equation}\label{eqn: S vp = X ZW}
S_\vp = X_{\clz, \clw}.
\end{equation}
To see this, first observe that $S_\vp^* = T_\vp^*|_{\clq}$. Also, by the reproducing property, we have
\[
T_\vp^* \Sk(\cdot, z_i) = \overline{\vp(z_i)} \Sk(\cdot, z_i),
\]
for all $i=1, \ldots, m$. Since  $\clq_\clz = \text{span} \{\Sk(\cdot, z_i): i=1, \ldots, m\}$, it then follows that
\[
S_\vp^* \Sk(\cdot, z_i) = \overline{\vp(z_i)} \Sk(\cdot, z_i),
\]
for all $i=1, \ldots, m$. Now, saying that $\vp$ interpolates $\clz$ and $\clw$ amounts to $\vp(z_i) = w_i$ for all $i=1, \ldots, m$. Hence $S_\vp=X_{\clz,\clw}$, proving \eqref{eqn: S vp = X ZW}. This yields the following interpretation: the data $\clz$ and $\clw$ are solved by an interpolant $\vp \in \cls(\B^n)$ if and only if $X_{\clz, \clw} \in \clb(\clq_{\clz})$ admits a lift, namely $X_{\clz, \clw} = S_\vp$. By Theorem \ref{contr lift}, we know that the operator $X_{\clz, \clw}$ on $\clq_{\clz}$ admits a lift if and only if the functional $X_{\clq_{\clz}}$ is a contraction on $(\clm_{\clq_{\clz}}, \|\cdot \|_1)$, where
\[
X_{\clq_{\clz}} f = \int_{\Sn} \psi f d\sigma.
\]
for all $f \in \clm_{\clq_{\clz}}$, and
\[
\psi = X_{\clz, \clw} (P_{\clq_{\clz}} 1).
\]
The result now follows from the fact that $\psi_{\clz, \clw} = \psi$, as observed in Proposition \ref{prop: basic Z and W}.
\end{proof}

For future reference, we record that, in the latter part of the proof of the above theorem, we actually proved that
\begin{equation}\label{eqn: X = Pi}
X_{\clq_{\clz}} = \Pi_{\clz, \clw},
\end{equation}
on $\clm_{\clq_{\clz}}$.

\begin{remark}
An inspection reveals that Theorem \ref{thm: NP interpolation contr} also holds, with the same proof, even without assuming $\clw \subset \D$.
\end{remark}

We shall frequently appeal to the observation in \eqref{eqn: S vp = X ZW} in the arguments that follow. By Proposition \ref{prop: basic Z and W}, we know that $\psi_{\clz, \clw} = \psi = X_{\clz, \clw} (P_{\clq_{\clz}} 1)$. This, together with Theorem \ref{lift dis} yields the following:

\begin{corollary}\label{thm: NP interpolation distance}
In the setting of Theorem \ref{thm: NP interpolation contr}, define
\[
\widetilde{\clm_{\clq_{\clz}}} = [\clq_{\clz}^{conj} \ominus \{\overline{\psi_{\clz, \clw}}\}] \dot+[\clm(\Sn) \dot+ H^2_0(\Sn)].
\]
Then $\clz$ and $\clw$ solve the interpolation problem if and only if
\[
\text{dist}_{L^1(\Sn)} \Big(\frac{\overline{\psi_{\clz, \clw}}}{\|\psi_{\clz, \clw}\|_2^2}, \widetilde{\clm_{\clq_{\clz}}}\Big) \geq 1.
\]
\end{corollary}

The above distance formula is reminiscent of the Nehari theorem \cite{Nehari}, although it arises in a different context, namely that of Hankel operators. Since relatively little is known about Hankel operators on $\B^n$, it may be worthwhile to explore whether these results can be connected to the classical theory of Hankel operators, should such a connection exist (see \cite{Arazy, Seip} for results on Hankel operators in several variables.)

A more refined analysis of the distance formula provides yet another characterization of interpolation. In the following, given a collection $\clw = \{w_i\}_{i=1}^m \subseteq \D$, we write $w = (w_1, \ldots, w_m) \in \mathbb{C}^m$, and use the standard notation:
\[
\|w\| = (\sum_{i=1}^{m}|w_i|^2)^{\frac{1}{2}}.
\]
Clearly $\clw \neq \{0\}$ if and only if $\|w\| \neq 0$.

\begin{corollary}\label{cor: interpol 2nd distance}
In the setting of Theorem \ref{thm: NP interpolation contr}, assume that $\|w\| \neq 0$, and define
\[
\eta_{\mathcal Z,\mathcal W} := \sum_{j=1}^m w_j \Sk(\cdot, {z_j}).
\]
Then $\clz$ and $\clw$ solve the interpolation problem if and only if
\[
\text{dist}_{L^1(\Sn)} \left(\frac{\overline{\eta_{\mathcal Z,\mathcal W}}}{\|w\|^2}, \widetilde{\clm_{\clq_{\clz}}} \right) \geq 1.
\]
\end{corollary}
\begin{proof}
We recall the reproducing property that $\langle \psi_{\clz, \clw}, \Sk(\cdot, {z_i})\rangle = \psi_{\clz, \clw}(z_i)$ for all $i = 1, \ldots, m$, and compute
\[
\begin{split}
\Big\langle \psi_{\clz, \clw} - \frac{\|\psi_{\clz, \clw}\|^2}{\|w\|^2} \eta_{\clz, \clw}, \psi_{\clz, \clw} \Big\rangle & = \|\psi_{\clz, \clw}\|^2 - \Big\langle \frac{\|\psi_{\clz, \clw}\|^2}{\|w\|^2} \eta_{\clz, \clw}, \psi_{\clz, \clw} \Big\rangle
\\
& = \|\psi_{\clz, \clw}\|^2 - \frac{\|\psi_{\clz, \clw}\|^2}{\|w\|^2} \Big \langle \sum_{j=1}^m w_j \Sk(\cdot, {z_j}), \psi_{\clz, \clw} \Big\rangle
\\
& = \|\psi_{\clz, \clw}\|^2 - \frac{\|\psi_{\clz, \clw}\|^2}{\|w\|^2}  \sum_{j=1}^m w_j \Big \langle \Sk(\cdot, {z_j}), \psi_{\clz, \clw} \Big\rangle
\\
& = \|\psi_{\clz, \clw}\|^2 - \frac{\|\psi_{\clz, \clw}\|^2}{\|w\|^2} \sum_{j=1}^m w_j \overline{\psi_{\clz, \clw}(z_j)}
\\
& = \|\psi_{\clz, \clw}\|^2 - \frac{\|\psi_{\clz, \clw}\|^2}{\|w\|^2} \sum_{j=1}^m |w_j|^2
\\
& = 0,
\end{split}
\]
where the penultimate equality follows from the fact that $\psi_{\mathcal Z,\mathcal W}(z_i)=w_i$ for all $i=1,\ldots,m$. This implies
\[
\overline{\psi_{\clz, \clw}} - \frac{\|\psi_{\clz, \clw}\|^2}{\|w\|^2} \overline{\eta_{\clz, \clw}} \in \clq_{\clz}^{conj} \ominus \{\overline{\psi_{\clz, \clw}}\}.
\]
At this point, we recall a basic property of the distance function in normed linear spaces. If $S$ is a closed subspace of a normed linear space $V$, then
\[
\text{dist}(x, S) = \|[x]\|_{V/S} = \inf \{\|x-y\|: y \in S\},
\]
where $[x]$ denotes the equivalence class of $x$ in the quotient space $V/S$, corresponding to the equivalence relation $s \sim t$ if and only if $s - t \in S$. From this perspective, the above observation suggests that
\[
\Big\|\Big[\overline{\psi_{\clz, \clw}}\Big]\Big\|_{L^1(\Sn)/\widetilde{\clm_{\clq_{\clz}}}} = \Big\|\Big[\frac{\|\psi_{\clz, \clw}\|^2}{\|w\|^2} \overline{\eta_{\clz, \clw}}\Big]\Big\|_{L^1(\Sn)/\widetilde{\clm_{\clq_{\clz}}}},
\]
and hence
\[
\text{dist}_{L^1(\Sn)} \left(\frac{\overline{\eta_{\mathcal Z,\mathcal W}}}{\|w\|^2}, \widetilde{\clm_{\clq_{\clz}}} \right) = \text{dist}_{L^1(\Sn)} \left(\frac{\overline{\psi_{\clz,\clw}}}{\|\psi_{\clz, \clw}\|_2^2}, \widetilde{\clm_{\clq_{\clz}}} \right).
\]
The result now follows from Corollary \ref{thm: NP interpolation distance}.
\end{proof}

We conclude this section by establishing a connection between interpolation theory and the class of inner functions $\cli(\B^n)$. Recall the Aleksandrov-Rudin approximation result in Theorem \ref{thm: intro Alek Rud}: Given $\vp \in \cls(\B^n)$, there exists a sequence $\{\vp_k\}_k \subseteq \cli(\B^n)$ such that
\[
\lim_{k \raro \infty} \vp_k = \vp,
\]
uniformly on compact subsets of $\B^n$. The same theorem also shows that this convergence holds in the weak$^*$-topology.

\begin{theorem}\label{thm: NP interpolation inner}
Let $\clz = \{z_i\}_{i = 1}^m \subseteq \B^n$ and $\clw = \{w_i\}_{i = 1}^m \subseteq \D$ be interpolation data. Then $\clz$ and $\clw$ solve the interpolation problem if and only if there exists a sequence of inner functions $\{\vp_k\}_k \subseteq \cli(\B^n)$ such that
\[
\lim_{k \raro \infty} \vp_k(z_j) = w_j,
\]
for all $j=1, \ldots, m$.
\end{theorem}
\begin{proof}
Suppose there is a function $\varphi \in \cls(\B^n)$ such that $\varphi(z_j) = w_j$ for all $j=1, \ldots, m$. By Theorem \ref{thm: intro Alek Rud}, there exists a sequence of inner functions $\{\vp_k\}_k$ such that
\[
\lim_{k \raro \infty} \vp_k = \vp,
\]
uniformly on compact subsets of $\B^n$. Since $\clz$ is a finite set and $\vp$ interpolates $\clz$ and $\clw$, the necessary part of the theorem follows. For the sufficient part, suppose there exists a sequence of inner functions $\{\vp_k\}_k$ such that
\[
\lim_{k \raro \infty} \vp_k (z_j) = w_j,
\]
for all $j=1, \ldots, m$. Viewing $\{\vp_k\}_k$ as a family of functions on $\B^n$, we observe that it is uniformly bounded, since $\|\vp_k\|_\infty=1$ for all $k$. Therefore, by Montel's theorem, there exist a function $\vp$ and a subsequence ${\vp_{k_\ell}}$ such that
\[
\lim_{l \raro \infty} \vp_{k_\ell} = \vp,
\]
uniformly on compact subsets of $\B^n$. As in the proof of Theorem $\ref{thm: extr sol}$, it follows that $\vp \in \cls(\B^n)$. Finally, for each $j = 1, \ldots, m$,
\[
\vp(z_j) = \lim_{l \raro \infty} \vp_{k_\ell}(z_j) = w_j.
\]
Therefore, $\vp \in \cls(\B^n)$ interpolates $\clz$ and $\clw$. After relabeling $\vp_{k_\ell}$ as $\vp_{\ell}$, the proof is complete.
\end{proof}

Another way to state the above result is as follows: $\clz$ and $\clw$ solve the interpolation problem if and only if there exists a sequence of inner functions $\{\vp_k\}_k \subseteq \cls(\B^n)$ such that
\[
\lim_{k \raro \infty} \vp_k(z_j) = \psi_{\clz, \clw}(z_j),
\]
for all $j=1, \ldots, m$.

Recall that, in one variable, by Pick's theorem (see Theorem \ref{thm: Intro Pick}), whenever an interpolation problem is solvable, there exists a finite Blaschke product that serves as an interpolant. Consequently, the above result is immediate when $n=1$. Clearly, for $n>1$, this theorem can be viewed as a higher-dimensional analogue of this classical fact, where the role of a finite Blaschke product is replaced by a sequence of inner functions.

We remark that, much like the proof of our commutant lifting theorem, where we invoked the Hahn–Banach theorem, the proof of the above inner function-based interpolation result also relies on another classical result, namely, Montel's theorem.

\section{The Carath\'{e}odory-Fej\'{e}r interpolation}\label{sec: Carath interp}

Recall the classical Carath\'{e}odory-Fej\'{e}r interpolation problem: Given scalars $c_0, c_1, \ldots, c_m$, find necessary and sufficient conditions for the existence of a Schur function $\vp \in \cls(\D)$ such that
\[
\frac{\vp^{(j)}(0)}{j!} = c_j,
\]
for all $j=0,1, \ldots, m$. The Carath\'{e}odory-Fej\'{e}r interpolation theorem states that such a Schur function exists if and only if the $(m+1) \times (m+1)$ matrix
\[
\begin{bmatrix}
c_0 & 0 & \cdots & 0
\\
c_1 & c_0 & \ddots & \vdots\\
\vdots & \ddots & \ddots & 0
\\
c_m & \cdots & c_1 & c_0
\end{bmatrix},
\]
is a contraction on $\mathbb{C}^{m+1}$. Although the problem was originally solved by Carath\'{e}odory \cite{Car 1, Car 2}, its matrix-theoretic formulation was developed by Toeplitz \cite{Toep}. Since then, the problem has been studied by numerous mathematicians in a variety of contexts, and the literature continues to grow. For instance, see Fejer \cite{Carat-Fej}, Fisher \cite{Fis}, Frobenius \cite{Frob}, Riesz \cite{R1, R2}, Schur \cite{Schur 1}, and Szeg\"{o} \cite{Szego}. We refer the reader to Sarason \cite{D Sarason} for a thorough historical account of the Carath\'{e}odory-Fej\'{e}r interpolation problem.

Similar to the interpolation problem, Sarason's lifting theorem approach also provides an alternative proof of this result. In this section, we apply the present lifting theorem to solve the same interpolation problem on $\B^n$. Specifically, we formulate the Carath\'{e}odory-Fej\'{e}r interpolation problem in the following way. The degree of a monomial $z^k$, $k = (k_1, \ldots, k_n) \in \Z_+^n$, is defined as
\[
\text{deg } z^k = |k| := \sum_{j=1}^{n} k_j.
\]
The degree of a polynomial is defined as the highest degree among its monomial terms.

\begin{problem}
Fix $m \geq 1$. Let $p \in \clq_m$, that is, $p \in \C[z_1, \dots, z_n]$ is a polynomial of degree less or equal to $m$. The Carath\'{e}odory-Fej\'{e}r interpolation problem seeks necessary and sufficient conditions for the existence of a Schur function $\varphi \in \cls(\B^n)$ such that
\[
p - \varphi \in \clq_m^\perp.
\]
\end{problem}

Given $m \in \mathbb{N}$, define the finite-dimensional quotient module $\clq_m$ as
\[
\clq_m = \{ p \in \C[z_1, \dots, z_n] : \text{deg } p \leq m\}.
\]

For a qualitative solution to this problem, we define a special subspace of $L^2(\Sn)$ as follows:
\[
\clm_m = \clq_m^{conj} \dot+ [\clm(\Sn) \dot+ H^2_0(\Sn)].
\]
As usual, we treat $\clm_m$ as a subspace of $L^1(\Sn)$.

\begin{theorem}\label{thm: Caratheodory contr}
Let $p \in \clq_m$. Define $X_p: \clm_m \raro \mathbb{C}$ by
\[
X_p f = \int_{\Sn} pf d\sigma,
\]
for all $f \in \clm_m$. Then there exists $\varphi \in \cls(\B^n)$ such that
\[
p - \varphi \in \clq_m^\perp,
\]
if and only if $X_p: (\clm_m, \|\cdot\|_1) \raro \C$ is a contraction.
\end{theorem}
\begin{proof}
Recall that $S_p = P_{\clq_m} T_p|_{\clq_m}$. Since $1 \in \clq_m$, we have
\[
\psi = S_p (P_{\clq_m} 1) = P_{\clq_m} T_p 1 = P_{\clq_m} p = p,
\]
and hence, by Theorem \ref{lift lemma}, $S_p$ admits a lift if and only if there exists a $\varphi \in \cls(\B^n)$ such that
\[
p - \varphi \in \clq_m^\perp.
\]
On the other hand, Theorem \ref{contr lift} implies that $S_p$ admits a lift if and only if $X_{\clq_m}: (\clm_{\clq_m}, \|\cdot\|_1) \raro \C$ is a contraction, where
\[
X_{\clq_m} f = \int_{\Sn} p f d\sigma,
\]
for all $f \in \clm_{\clq_m}$. Therefore, the result follows because $X_{\clq_m} = X_p$.
\end{proof}

Analogous to the solution of the interpolation theorem in Theorem \ref{thm: NP interpolation distance}, the Carath\'{e}odory-Fej\'{e}r interpolation theorem can also be formulated in terms of a distance formula, yielding a quantitative characterization of the solutions. The proof is similar to that of the quantitative characterization in Corollary \ref{thm: NP interpolation distance}, and hence we omit the details.

\begin{theorem}\label{thm: CF interpolation distance}
Let $p$ be a nonzero element in $\clq_m$. Define
\[
\widetilde{\clm_{p}} = [\clq_m^{conj} \ominus \{\bar{p}\}] \dot+ [\clm(\Sn) \dot+ H^2_0(\Sn)].
\]
Then there exists $\varphi \in \cls(\B^n)$ such that
\[
p- \vp \in \clq_m^\perp,
\]
if and only if
\[
\text{dist}_{L^1(\Sn)}\Big(\frac{\overline{p}}{\|p\|_2^2}, \widetilde{\clm_{p}}\Big) \geq 1.
\]
\end{theorem}

Now we turn to $\cla(\B^n)$. Recall that the class $\cla(\B^n)$ of approximate functions refers to either the class $\cli(\B^n)$ of inner functions or the class $\cle(\B^n)$ of extremal functions.

\begin{theorem}\label{thm: Caratheodory inner}
Let $p$ be a nonzero element in $\clq_m$. Then there exists $\varphi \in \cls(\B^n)$ such that
\[
p - \varphi \in \clq_m^\perp,
\]
if and only if there is a sequence of functions $\{\vp_k\}_k \subseteq \cla(\B^n)$ such that
\[
\lim_{k \raro \infty} \la \vp_k, z^l \ra = \la p , z^l \ra,
\]
for all $l \in \Z_{+}^n$ with $|l| \leq m$.
\end{theorem}
\begin{proof}
In the proof of Theorem \ref{thm: Caratheodory contr}, we have already noted that there exists a $\varphi \in \cls(\B^n)$ such that $p - \varphi \in \clq_m^\perp$ if and only if the module map $S_p \in \clb(\clq_m)$ admits a lift. On the other hand, by Corollary $\ref{Cor: lift inner q}$, $S_p$ admits a lift if and only if there exists a sequence $\{\vp_k\}_k \subseteq \cls(\B^n)$ such that
\be
w^* - \lim_{k \raro \infty} \vp_k = p \text{ on } \clq_m^{conj}.
\ee
In view of the fact that $\clq_m$ is spanned by the monomials $\{z^l : |l| \leq m\}$, this is equivalent to
\[
\lim_{k \raro \infty} \int_{\Sn} \vp_k \overline{z^l} d\sigma = \int_{\Sn} p \overline{z^l} d\sigma.
\]
that is, $\lim_{k \raro \infty}  \la \vp_k, z^l \ra = \la p , z^l\ra$ for all $l\in \Zp$ with $|k| \leq m$.
\end{proof}

To further elaborate this result, we recall that if $f$ on $\B^n$ is an analytic function, then $f$ admits a power series representation $f(z) = \sum_k a_k z^k$ on $\B^n$, where (note that the set of analytic monomials $\{z^k: k \in \Zp\}$ is an orthogonal set in $H^2(\B^n)$, see \cite[1.4.8]{Rud})
\[
a_k = \frac{\la f, z^k\ra}{\la z^k, z^k\ra},
\]
as well as
\[
a_k = \frac{1}{k!} \frac{\partial^k f}{\partial z^k}\Big|_0,
\]
for all $k \in \Z_+^n$.

Using the above observations, the Carath\'{e}odory-Fej\'{e}r interpolation theorem can be reformulated in terms of the coefficients in the power series expansion of the corresponding functions as follows:

\begin{corollary}\label{CF: Taylor}
Let $p \in \clq_m$. Then there exists $\varphi \in \cls(\B^n)$ such that
\[
p - \varphi \in \clq_m^\perp,
\]
if and only if there is a sequence of functions $\{\vp_k\}_k \subseteq \cla(\B^n)$ such that
\[
\lim_{k \raro \infty} \frac{\partial^{\alpha} \vp_k}{\partial z^{\alpha}}\Big|_0  = \frac{\partial^{\alpha} p}{\partial z^{\alpha}} \Big|_0,
\]
for all $\alpha \in \Z_{+}^n$ with $|\alpha| \leq m$.
\end{corollary}

As in the proof of Theorem \ref{thm: NP interpolation inner}, one may also prove Theorem \ref{thm: Caratheodory inner} using Montel's theorem.

\section{Norm-preserving lifting: $n>1$}\label{sec: norm pres lift}

In this section, we examine a phenomenon specific to the higher-dimensional case $n>1$ in the context of norm-preserving liftings. We begin by recalling the classical commutant lifting theorem of Sarason \cite{D Sarason}:

\begin{theorem}[Sarason. $n=1$]\label{thm: Sarason}
Let $\clq$ be a quotient module of $H^2(\D)$, and let $X \in \clb_1(\clq)$ be a module map. Then there exists $\vp \in \cls(\D)$ such that $X = S_\vp$ and
\[
\|X\| = \|\vp\|_\infty.
\]
\end{theorem}

In other words, norm-preserving lifting always holds in the $n=1$ case. Our approach to lifting in Theorem \ref{contr lift}, when restricted to $n=1$, recovers precisely Sarason’s result, as in the polydisc case treated in \cite[Theorem 9.2]{KD}. In this case, however, one obtains only the inequality
\[
\|X\| \leq \|\vp\|_\infty,
\]
and not necessarily the norm-preserving aspect of lifting.

In stark contrast to Sarason, we show in this section that such a norm-preserving lifting phenomenon always fails for finite-dimensional quotient modules of $H^2(\B^n)$ whenever $n>1$. First, we prove a lemma of general interest. We follow the notation introduced in Theorem \ref{contr lift}.

\begin{lemma}\label{lemma: norm X, Xq}
Let $n \geq 1$. Let $\clq \subseteq H^2(\B^n)$ be a quotient module and let $X \in \clb(\clq)$ be a module map. Define
\[
\psi = X(P_\clq 1),
\]
and define $X_\clq : (\clm_\clq, \|\cdot\|_1) \raro \C$ by
\[
X_\clq f = \int_{\Sn} \psi f d\sigma,
\]
for all $f \in \clm_\clq$. If $X_\clq$ is bounded, then
\[
\|X\| \leq \|X_{\clq}\|.
\]
\end{lemma}
\begin{proof}
First, suppose that $\|X\| = 1$. If $\|X_{\clq}\| < 1$, then Theorem $\ref{contr lift}$ implies that $X$ admits a lift, and (see \eqref{eqn: |theta| = XQ 1}) there exists a function $\theta \in \cls(\B^n)$ such that $X = S_{\theta}$ and
\[
\|\theta\|_{\infty} = \|X_{\clq}\| < 1.
\]
For $f \in \clm_\clq$, we have
\[
\|X f\|_2 = \|S_{\theta} f\|_2 = \|P_{\clq} \theta f\|_2 \leq \|\theta f\|_2 \leq \|\theta\|_{\infty} \|f\|_2,
\]
and hence $\|X\| \leq \|\theta\|_{\infty} <1$, which contradicts the assumption $\|X\| = 1$. This implies $\|X_{\clq}\| \geq 1$, that is, $\|X\| \leq \|X_{\clq}\|$.

\noindent For the general case, assume that $X$ is a nonzero module map, and define
\[
A = \frac{1}{\|X\|} X.
\]
Clearly, $A \in \clb(\clq)$ is a module map and $\|A\| = 1$. By the above case, we have
\[
1 = \|A\| \leq \|A_{\clq}\|.
\]
Now, for any $f \in \clm_{\clq}$,
\[
A_{\clq} f = \int_{\Sn} (A P_{\clq} 1) f d\sigma = \int_{\Sn} \frac{1}{\|X\|} X(P_{\clq} 1) f d\sigma = \frac{1}{\|X\|} X_{\clq} f,
\]
which implies
\[
A_{\clq} = \frac{1}{\|X\|} X_{\clq}.
\]
Since $\|A_{\clq}\| \geq 1$, it follows that
\[
\frac{\|X_{\clq}\|}{\|X\|} \geq 1,
\]
that is, $\|X\| \leq \|X_{\clq}\|$. This completes the proof of the lemma.
\end{proof}

In the context of the above lemma, note that if $\clq$ is finite-dimensional, then $X_\clq$ is automatically bounded. We also point out a property of finite-dimensional quotient modules: If $\clq \subseteq H^2(\B^n)$ is a finite-dimensional quotient module, then (cf. \cite[Corollary 3.4]{GZ} and  \cite[Chapter 2, Remark 2.5.5]{CG})
\begin{equation}\label{eqn: fd Q in A}
\clq \subset A(\B^n) \subset H^{\infty}(\B^n).
\end{equation}

Before proceeding to the next result, we recall an interesting result of Rudin that highlights a significant distinction between the cases $n=1$ and $n>1$ \cite[Theorem 4.6]{Rud inner}: Let $n>1$. If $f$ is a nonzero function in $A(\B^n)$ and $u \in H^\infty(\B^n)$ is a nonconstant inner function, then
\begin{equation}\label{eqn: Rudin Smirnov}
\frac{f}{u} \notin H^2(\B^n).
\end{equation}
In fact, more is true: $\frac{f}{u}$ is not even in the Smirnov class $N_*(\B^n)$. In the following, we bring together Rudin's result with our commutant lifting theorem and the above lemma:

\begin{theorem}\label{thm: norm 1 no lift}
Assume $n > 1$. Let $\clq \subseteq H^2(\B^n)$ be a finite-dimensional quotient module and let $X \in \clb(\clq)$ be a module map with $\|X\| = 1$. Then $X$ admits a lift if and only if there exists a unimodular constant $\alpha$ such that
\[
X = \alpha I_\clq.
\]
\end{theorem}
\begin{proof}
As in Lemma \ref{lemma: norm X, Xq}, define $X_\clq : (\clm_\clq, \|\cdot\|_1) \raro \C$ by
\[
X_\clq f = \int_{\Sn} (X (P_\clq 1)) f d\sigma,
\]
for all $f \in \clm_\clq$. Since $\clq$ is finite-dimensional, we know that $X_\clq$ is continuous. Moreover, as $\|X\| = 1$, Lemma \ref{lemma: norm X, Xq} implies
\[
\|X_{\clq}\| \geq 1.
\]
Suppose $X$ admits a lift. Theorem $\ref{contr lift}$ ensures that $\|X_{\clq}\| \leq 1$ and thus we must have
\[
\|X_{\clq}\| = 1.
\]
By the sufficient part of the proof of Theorem \ref{contr lift}, and more specifically, by \eqref{eqn: |theta| = XQ 1}, there exists $\theta \in \cls(\B^n)$ such that $\|\theta\|_{\infty} = 1$ and
\[
X = S_{\theta}.
\]
Given that $\clq$ is finite-dimensional, we conclude that $X \in \clb(\clq)$ is norm-attaining: there exists $f \in \clq$, $\|f\|_2 = 1$, such that $\|Xf\|_2 = 1$. Hence
\[
1 = \|Xf\|_2 = \|S_{\theta} f\|_2 = \|P_{\clq} \theta f\|_2 \leq \|\theta f\|_2 \leq 1.
\]
Therefore, the above chain of inequalities reduces to equalities. We conclude that
\[
P_{\clq} \theta f = \theta f,
\]
and, in particular,
\[
\theta f \in \clq.
\]
Consequently,
\[
Xf = S_\theta f = P_\clq(\theta f) = \theta f.
\]
We claim that $\theta$ is a constant. We proceed as follows. Since
\[
1 = \|Xf\|_2 = \|\theta f\|_2 = \|f\|_2,
\]
it follows that
\[
1 = \int_{\Sn} |f|^2 d\sigma = \int_{\Sn} |\theta f|^2 d\sigma,
\]
that is,
\[
\int_{\Sn} (1 - |\theta|^2)|f|^2 d\sigma = 0.
\]
Since $\theta \in \cls(\B^n)$, the integrand is nonnegative. Hence,
\[
(1 - |\theta|^2)|f|^2 = 0,
\]
on $\Sn$ a.e. Define
\[
E_1 = \{\zeta \in \Sn : ((1 - |\theta|^2)|f|^2)(\zeta) \neq 0\}.
\]
Clearly,
\[
\sigma(E_1) = 0.
\]
Define
\[
E_2 = \{\zeta \in \Sn : f(\zeta) = 0\}.
\]
Given that $f$ is a nonzero analytic function in $\clq \subseteq H^2(\B^n)$, we have (in view of the $K$-limits, cf. \cite[Theorem 5.5.9]{Rud})
\[
\sigma(E_2) = 0.
\]
Let $E = E_1 \cup E_2$. Then
\[
\sigma(E) = 0.
\]
For any $\zeta \in E^{c}$, we have $f(\zeta) \neq 0$ and $((1 - |\theta|^2)|f|^2)(\zeta) = 0$, which implies
\[
|\theta|(\zeta) = 1,
\]
for $\zeta \in E^{c}$. Since $\sigma(E) = 0$ as well, we conclude that $\theta$ is inner. Since $\clq \subseteq H^2(\B^n)$ is finite-dimensional, \eqref{eqn: fd Q in A} implies
\[
\clq \subset A(\B^n) \subset H^{\infty}(\B^n).
\]
As noted earlier,
\[
\theta f = Xf \in \clq \subseteq A(\B^n)
\]
and hence $\theta f \in A(\B^n)$. On the other hand, writing $f = \frac{\theta f}{\theta}$, we conclude
\[
\frac{\theta f}{\theta} = f \in H^2(\B^n).
\]
If $\theta$ were nonconstant, this would contradict Rudin \cite[Theorem 4.6]{Rud inner} (or see the discussion preceding the statement of this theorem). Hence $\theta \equiv \alpha$ for some $\alpha \in \mathbb{T}$. Since $X = S_{\theta}$, it follows that $X = \alpha I_\clq$. The converse is immediate.
\end{proof}

As a result, for $n > 1$, we obtain the following result, which stands in sharp contrast to Sarason's conclusion in the one-variable case $n=1$:

\begin{corollary}\label{cor: norm preserving}
Assume $n > 1$. Let $X$ be a nonconstant module map acting on a finite-dimensional quotient module $\clq \subseteq H^2(\B^n)$. Then $X$ does not admit a norm-preserving lift.
\end{corollary}
\begin{proof}
Let $A= \frac{X}{\|X\|}$. Then $A \in \clb(\clq)$ is a module map with $\|A\| = 1$. Suppose $X$ has a norm-preserving lift, say $X = S_{\theta}$, $\theta \in \cls(\B^n)$, with $\|\theta\|_{\infty} = \|X\|$. Define
\[
\varphi = \frac{1}{\|X\|}\theta.
\]
Clearly, $\vp \in \cls(\B^n)$ is a nonconstant function and $\|\vp\|_\infty = 1$. For $f \in \clq$, we have
\[
Af = \frac{1}{\|X\|} P_{\clq} \theta f = P_{\clq} \varphi f.
\]
This shows that $A$ admits a lift, namely, $S_\vp$ - which contradicts Theorem \ref{thm: norm 1 no lift}.
\end{proof}

However, quiet interestingly, module maps on infinite-dimensional quotient modules may admit norm-preserving lifts, as we will see in Theorem \ref{lifiting inner multiple}.

We now assume that $n \geq 1$. The following result stands in contrast to the previous corollary:

\begin{theorem}\label{lemma: Xq X norm}
Let $\clq \subseteq H^2(\B^n)$ be a quotient module, and let $X \in \clb_1(\clq)$ be a module map. Then $X$ admits a norm preserving lift if and only if
\[
\|X\| = \|X_{\clq}\|.
\]
\end{theorem}
\begin{proof}
Suppose that $X$ admits a norm-preserving lift:
\[
X = S_{\varphi} \text{ with }\|X\| = \|\varphi\|_{\infty},
\]
for some $\varphi \in \cls(\B^n)$. By Theorem \ref{contr lift} we know that $X_{\clq}$ is bounded. Then, by Theorem \ref{thm: norm XQ}, we have
\[
\|X_\clq\| \leq \|\vp\|_\infty = \|X\|.
\]
Together with Lemma \ref{lemma: norm X, Xq}, this implies that $\|X\| = \|X_{\clq}\|$. Conversely, suppose that $\|X\| = \|X_{\clq}\|$. Since $\|X\| \leq 1$, we have $\|X_{\clq}\| \leq 1$. By Theorems \ref{contr lift} and \ref{thm: norm XQ}, we conclude that $X$ admits a lift:
\[
X = S_{\varphi} \text{ with }\|X_{\clq}\| = \|\varphi\|_{\infty},
\]
for some $\varphi \in \cls(\B^n)$. Since $\|X\| = \|X_{\clq}\|$, it follows that $X$ admits a norm-preserving lift.
\end{proof}

Combining this with Corollary \ref{cor: norm preserving}, we have the following:

\begin{corollary}[On $n > 1$]\label{lemma: Xq less X}
Assume $n > 1$. Let $\clq$ be a finite-dimensional quotient module of $H^2(\B^n)$ and let $X \in \clb_1(\clq)$ be a nonconstant module map. Then
\[
\|X\| < \|X_{\clq}\|.
\]
\end{corollary}
\begin{proof}
If $\|X\| = \|X_{\clq}\|$, then by Theorem \ref{lemma: Xq X norm}, $X$ admits a norm preserving lift. This contradicts Corollary \ref{cor: norm preserving}, which says that $X$ does not admit a norm-preserving lift. Therefore, by Lemma \ref{lemma: norm X, Xq}, we conclude that $\|X\| < \|X_{\clq}\|$.
\end{proof}

However, somewhat curiously, in the case $n=1$, the above inequality becomes an equality:

\begin{corollary}[On $n = 1$]\label{remark: single variable xq x}
Let $\clq$ be a quotient module of $H^2(\D)$, and let $X \in \clb_1(\clq)$ be a nonconstant module map.  Then
\[
\|X\| = \|X_{\clq}\|.
\]
\end{corollary}
\begin{proof}
Define
\[
A = \frac{1}{\|X\|} X
\]
Then $A \in \clb(\clq)$ is a module map with $\|A\| = 1$. Suppose $A = S_{\theta}$ for some $\theta \in \cls(\D)$. For each $f \in \clq$, we have
\[
\|Af\|_2 = \|S_{\theta}f\|_2 = \|P_{\clq} \theta f\|_2 \leq \|\theta f\|_2 \leq \|\theta\|_{\infty} \|f\|_2,
\]
that is, $\|A\| \leq \|\theta\|_{\infty}$. Therefore,
\[
\|A\| \leq \inf \{\|\theta\|_{\infty} : A = S_{\theta}, \theta \in \cls(\D)\}.
\]
On the other hand, by Sarason, Theorem \ref{thm: Sarason}, there exists $\varphi \in \cls(\D)$ such that
\[
A = S_{\varphi}, \text{ and } \|\varphi\|_{\infty} = \|A\| = 1,
\]
and hence
\[
\|A\| = \inf \{\|\theta\|_{\infty} : A = S_{\theta}, \theta \in \cls(\D)\}.
\]
Moreover, by the norm formula in Theorem \ref{thm: norm XQ},
\[
\|A_{\clq}\| = \inf \{\|\theta\|_{\infty} : A = S_{\theta}, \theta \in \cls(\D)\},
\]
where $A_\clq: (\clm_\clq, \|\cdot\|_1) \raro \mathbb{C}$ given by
\[
A_\clq f = \int_{\mathbb{T}} f \psi d\,\sigma,
\]
for all $f \in \clm_\clq$ with $\psi = A (P_\clq 1)$. Therefore
\[
\|A_{\clq}\| = \|A\| = 1.
\]
Since $A (P_{\clq} 1) = \frac{1}{\|X\|} X (P_{\clq}1)$, it follows that $A_{\clq} = \frac{1}{\|X\|} X_{\clq}$, and hence $\|X\| = \|X_{\clq}\|$.
\end{proof}

Summarizing the preceding two results, we obtain the following dichotomy between the cases $n=1$ and $n> 1$: Let $\clq$ be a finite-dimensional quotient module of $H^2(\B^n)$, and let $X \in \clb_1(\clq)$ be a nonconstant module map. Then
\[
\|X\|  = \|X_{\clq}\|  \mbox{ if } n=1,
\]
and
\[
\|X\|  < \|X_{\clq}\| \mbox{ if } n > 1.
\]

\newsection{Extremal problem and its solution}\label{Sec: ext inter}

This section provides characterizations of the extremal problem. We also present a useful construction of interpolation data that solve the extremal problem. This construction ultimately yields a factorization of a solution to the interpolation problem as the product of an extremal function and the minimum-norm scalar factor among all solutions. See Section \ref{sec: Carath approx} for the definition and other basic details regarding the extremal problem. Since the extremal problem is the main focus of this section, we shall assume throughout that the number of interpolation nodes is greater than one; that is,
\[
m>1.
\]

Recall that, when $n=1$, the interpolation data $\clz = \{z_i\}_{i=1}^m \subseteq \D$ and $\clw = \{w_i\}_{i=1}^m \subseteq \D$ solve the extremal interpolation problem if and only if the associated Pick matrix is positive semi-definite and of low rank \cite[Chapter 6]{AM}. The latter condition is equivalent to
\[
\det \begin{bmatrix}\frac{1 - w_i \overline{w_j}}{1 - z_i \overline{z_j}}\end{bmatrix}_{m \times m} = 0.
\]
Moreover, in this case, the extremal solution is unique and is a finite Blaschke product \cite[Theorem 6.4]{AM}. In the case of $\B^n$, we give concrete characterizations of this problem using the quantitative and qualitative criteria studied earlier.

Fix interpolation data $\clz = \{z_i\}_{i=1}^m \subseteq \B^n$ and $\clw = \{w_i\}_{i=1}^m \subseteq \D$. As in Section \ref{sec: interpolation}, or more specifically in Theorem \ref{thm: NP interpolation contr}, we the define quotient module
\[
\clq_{\clz} = \text{span}\{\Sk(\cdot, z_i) : i=1, \ldots, m\},
\]
and the space
\[
\clm_{\clq_{\clz}} = \clq_{\clz}^{conj} \dot+ [\clm(\Sn) \dot+ H^2_0(\Sn)].
\]
Recall from \eqref{eqn: X* zw} that $X_{\clz, \clw} \in \clb(\clq_\clz)$ is a module map satisfying
\[
X^*_{\clz, \clw} \Sk(\cdot, z_j) = \overline{{w}_j } \Sk(\cdot, z_j),
\]
for all $j=1, \ldots, m$. Moreover, we have the functional $\Pi_{\clz, \clw}: (\clm_{\clq_{\clz}}, \|\cdot\|_1) \raro \C$ defined by
\[
\Pi_{\clz, \clw} f = \int_{\Sn} \psi_{\clz, \clw} f d\sigma,
\]
for all $f \in \clm_{\clq_{\clz}}$, where (see Theorem \ref{thm: NP interpolation contr})
\[
\psi_{\clz, \clw} = \sum_{j = 1}^m c_j \Sk(\cdot, z_j),
\]
with
\[
\begin{bmatrix}
c_1
\\
\vdots
\\
c_m
\end{bmatrix}
=
\begin{bmatrix}
\Sk(z_i, z_j)
\end{bmatrix}_{m \times m}^{-1}
\begin{bmatrix}
w_1
\\
\vdots\\
w_m
\end{bmatrix}.
\]
Finally, recall from Corollary \ref{thm: NP interpolation distance} that
\[
\widetilde{\clm_{\clq_\clz}} = \ker \Pi_{\clz, \clw} = [\clq_{\clz}^{conj} \ominus \{\overline{\psi_{\clz, \clw}}\}] \dot+ [\clm(\Sn) \dot+ H^2_0(\Sn)].
\]

\begin{theorem}\label{thm: extremal}
Let $\clz = \{z_i\}_{i=1}^m \subseteq \B^n$ and $\clw = \{w_i\}_{i=1}^m \subseteq \D$ be interpolation data. Then the following are equivalent:
\begin{enumerate}
\item $\clz$ and $\clw$ solve the extremal problem.
\item $\|\Pi_{\clz, \clw}\| = 1$.
\item $\text{dist}_{L^1(\Sn)}\Big(\frac{\overline{\psi_{\clz, \clw}}}{\|\psi_{\clz, \clw}\|_2^2}, \widetilde{\clm_{\clq_{\clz}}}\Big) = 1$.
\item $\text{dist}_{L^1(\Sn)} \left(\frac{1}{\sum_{i=1}^{m}|w_i|^2} \displaystyle \sum_{j=1}^m \overline{w_j} \overline{\Sk(\cdot, {z_j})}, \widetilde{\clm_{\clq_{\clz}}} \right) = 1$.
\end{enumerate}
\end{theorem}
\begin{proof}
First, we prove that (1) implies (2). Suppose $\clz$ and $\clw$ solve the extremal problem. This means, in particular, that there exists an interpolant. By Theorem \ref{thm: NP interpolation contr}, we have
\[
\|\Pi_{\clz, \clw}\| \leq 1.
\]
Assume, for contradiction, that $\|\Pi_{\clz, \clw}\| < 1$. Then, by Theorem \ref{contr lift} and Theorem \ref{thm: norm XQ}, there exists $\theta \in \cls(\B^n)$ such that $X_{\clz, \clw} = S_\theta$ and
\[
\|\theta\|_{\infty} = \|\Pi_{\clz, \clw}\| < 1.
\]
Since $X_{\clz, \clw} = S_\theta$, by \eqref{eqn: S vp = X ZW}, we know that $\theta$ interpolates $\clz$ and $\clw$. However, since the interpolation problem is extremal, this contradicts the definition. Therefore, we must have $\|\Pi_{\clz, \clw}\| = 1$.

\noindent To prove that (2) implies (1), assume that $\|\Pi_{\clz, \clw}\| = 1$. Since $\Pi_{\clz, \clw}$ is a contraction, by Theorem \ref{contr lift}, Theorem \ref{thm: norm XQ}, and Theorem \ref{thm: NP interpolation contr}, the interpolation problem has a solution $\theta \in \cls(\B^n)$ with $\|\theta\|_{\infty} = 1$. Now, suppose $\varphi \in \cls(\B^n)$ interpolates $\clz$ and $\clw$ and satisfies
\[
\|\varphi\|_{\infty} < 1.
\]
By \eqref{eqn: S vp = X ZW}, we have $X_{\clz, \clw} = S_\vp$. Now, in view of Theorem \ref{thm: norm XQ}, we have
\[
\|\Pi_{\clz, \clw}\| = \inf \{\|\theta\|: \theta \in \cls(\B^n), S_\theta = X_{\clz, \clw}\}.
\]
Thus,
\[
\|\Pi_{\clz, \clw}\| \leq \|\vp\|_\infty < 1,
\]
which contradicts the assumption that $\|\Pi_{\clz, \clw}\| = 1$. Therefore, the interpolation problem is extremal. The equivalence of (2) and (3) follows from Lemma \ref{dis norm} in exactly the same way as in the proof of Theorem \ref{lift dis}. For the equivalence of (3) and (4), it suffices to recall from the proof of Theorem \ref{cor: interpol 2nd distance}, specifically, the final identity in the proof, that
\[
\text{dist}_{L^1(\Sn)} \left(\frac{1}{\sum_{i=1}^{m}|w_i|^2} \displaystyle \sum_{j=1}^m \overline{w_j} \overline{\Sk(\cdot, {z_j})}, \widetilde{\clm_{\clq_{\clz}}} \right) = \text{dist}_{L^1(\Sn)}\Big(\frac{\overline{\psi}}{\|\psi\|_2^2}, \widetilde{\clm_{\clq_{\clz}}}\Big).
\]
This completes the proof of the theorem.
\end{proof}

Next, we formulate an extremal problem based on a given solvable interpolation problem. This approach will ultimately lead to a solution for the original solvable interpolation problem.

\begin{theorem}\label{thm: int vs ext}
Suppose $\clz = \{z_i\}_{i=1}^m \subset \B^n$ and $\clw = \{w_i\}_{i=1}^m \subset \D$ are interpolation data, with the $w_i$ not all equal. If $\clz$ and $\clw$ solve the interpolation problem, then $\clz$ and $\widehat{\clw}$ solve the extremal problem, where
\[
\widehat{\clw} = \{\widehat{w}_1, \ldots, \widehat{w}_m\} \subset \D,
\]
and
\[
\widehat{w}_j = \frac{1}{\|\Pi_{\clz, {\clw}}\|} w_j,
\]
for all $j=1, \ldots, m$.
\end{theorem}
\begin{proof}
Note that $\Pi_{\clz, \clw}$ is a nonzero bounded linear functional on $\clm_{\clq_\clz}$. Since the interpolation problem is solvable, Theorem \ref{thm: NP interpolation contr} implies that
\[
\|\Pi_{\clz, \clw}\| \leq 1.
\]
First, we claim that
\[
\widehat{\clw} = \{\widehat{w}_1, \ldots, \widehat{w}_m\} \subseteq \D,
\]
where
\[
\widehat{w}_j = \frac{1}{\|\Pi_{\clz, \clw}\|} w_j.
\]
for $j = 1, \ldots, m$. Indeed, by Proposition \ref{prop: basic Z and W}, we know that
\[
\psi_{\clz, \clw} = \psi,
\]
where $\psi = X_{\clz, \clw}(P_{\clq_{\clz}}1)$, and $X_{\clz, \clw} \in \clb(\clq_{\clz})$ is the module map defined by \eqref{eqn: X* zw}, namely,
\[
X_{\clz, \clw}^* \Sk(\cdot, z_j) = \overline{w_j} \Sk(\cdot, z_j),
\]
for all $j=1, \ldots, m$. Moreover, \eqref{eqn: X = Pi} implies that
\[
X_{\clq_{\clz}} = \Pi_{\clz, \clw},
\]
where
\[
X_{\clq_{\clz}} f = \int_{\Sn} \psi f d\sigma,
\]
for all $f \in \clm_{\clq_{\clz}}$. Then
\[
\|X_{\clq_{\clz}}\| = \|\Pi_{\clz, \clw}\| \leq 1,
\]
that is, the functional $X_{\clq_{\clz}}: (\clm_{\clq_{\clz}}, \|\cdot\|_1) \raro \mathbb{C}$ is a contraction. By the sufficiency part of Theorem \ref{contr lift}, more specifically, by \eqref{eqn: |theta| = XQ 1}, there exists a Schur function $\vp \in \cls(\B^n)$ such that
\[
X_{\clz, \clw} = S_\vp,
\]
and
\[
\|X_{\clq_{\clz}}\| = \|\vp\|_\infty.
\]
The latter identity implies that
\[
\|\Pi_{\clz, \clw}\| = \|\vp\|_{\infty} \leq 1,
\]
where the first identity, $X_{\clz, \clw} = S_{\vp}$, together with the proof of Theorem \ref{thm: NP interpolation contr} (more specifically, \eqref{eqn: S vp = X ZW}), implies that $\vp$ interpolates the data $\clz$ and $\clw$:
\[
\vp(z_i) = w_i,
\]
for all $i=1, \ldots, m$. In particular, for each $j = 1, \ldots, m$, we have
\[
\|\Pi_{\clz, \clw}\| = \|\vp\|_{\infty} = \sup_{z \in \B^n} |\vp(z)| \geq |\vp(z_j)| = |w_j|,
\]
that is,
\[
|\widehat{w}_j| = \frac{|w_j|}{\|\Pi_{\clz, \clw}\|} \leq 1.
\]
To obtain a contradiction, suppose there exists $k \in \{1, \ldots, m\}$ such that
\[
|\widehat{w_k}| = 1.
\]
By the definition of $\widehat{w_k}$, this means,
\[
|w_k| = \|\Pi_{\clz, \clw}\|,
\]
and hence
\[
|w_k| = \|\Pi_{\clz, \clw}\| = \|\vp\|_{\infty}.
\]
This implies
\[
|\vp(z_k)| = |w_k| = \|\vp\|_{\infty}.
\]
By the maximum modulus principle, $\vp$ must be constant; a contradiction, since not all $w_j$ are equal. Thus, we obtain $\widehat{w}_j \in \D$, for all $j = 1, \dots ,m$. In other words, we have
\[
\widehat \clw = \{\widehat{w}_1, \ldots, \widehat{w}_m\} \subseteq \D.
\]
Now we claim that $\clz$ and $\widehat\clw$ solve the extremal problem. Along with the notation of Theorem \ref{thm: NP interpolation contr}, we proceed as follows. Define module map $X_{\clz, \widehat{\clw}} \in \clb(\clq_\clz)$ by
\[
X_{\clz, \widehat{\clw}}^* \Sk(\cdot, z_j) = \overline{\widehat{w}_j} \Sk(\cdot, z_j).
\]
that is,
\[
X_{\clz, \widehat{\clw}}^* \Sk(\cdot, z_j) = \frac{1}{\|\Pi_{\clz, \clw}\|} X_{\clz, \clw}^* \Sk(\cdot, z_j),
\]
for all $j=1, \ldots, m$. In other words,
\[
X_{\clz, \widehat{\clw}} = \frac{1}{\|\Pi_{\clz, \clw}\|} X_{\clz, \clw}.
\]
As usual, corresponding to the data $\clz$ and $\widetilde{\clw}$, define
\[
\widehat{\psi} = X_{\clz, \widehat{\clw}}(P_{\clq_{\clz}} 1).
\]
Recall that
\[
\psi = X_{\clz, \clw} (P_{\clq_\clz} 1),
\]
and hence
\[
\psi = \|\Pi_{\clz, \clw}\| X_{\clz, \widehat{\clw}} (P_{\clq_{\clz}} 1) = \|\Pi_{\clz, \clw}\| \widehat{\psi}.
\]
Similarly, define the functional $\Pi_{\clz, \widehat{\clw}}: (\clm_{\clq_\clz}, \|\cdot\|_1) \raro \mathbb{C}$ by
\[
\Pi_{\clz, \widehat{\clw}} f = \int_{\Sn} \widehat{\psi} f d\sigma,
\]
for all $f \in \clm_\clq$. Therefore,
\[
\Pi_{\clz, \widehat{\clw}} f = \frac{1}{\|\Pi_{\clz, \clw}\|} \Pi_{\clz, \clw} f,
\]
for all $f \in \clm_{\clq_{\clz}}$, and hence
\[
\|\Pi_{\clz, \widehat{\clw}}\| = 1.
\]
Consequently, Theorem \ref{thm: extremal} implies that $\clz$ and $\widehat{\clw}$ solve the extremal problem.
\end{proof}

Note that the assumption that the $w_i$ are not all equal in the above is essential. Indeed, if they are all equal, then the construction of $\widehat{\clw}$ yields a singleton set. Consequently, the interpolation data $\clz$ and $\widehat{\clw}$ do not solve the extremal problem.

The above result has an important consequence. For each $m \geq 1$, define (see Section \ref{sec: Carath approx})
\[
\cle_m(\B^n) = \{\vp \in \cls(\B^n): \vp \text{ is a solution to some } m \text{-point extremal problem}\}.
\]

\begin{corollary}\label{cor: int vs ext}
Suppose $\clz = \{z_i\}_{i=1}^m \subset \B^n$ and $\clw = \{w_i\}_{i=1}^m \subset \D$ are interpolation data, with the $w_i$ not all equal. Suppose that $\clz$ and $\clw$ solve the interpolation problem. Define
\[
\lambda_{\clz, \clw} = \inf \{\|\eta\|_\infty: \eta \in \cls(\B^n) \text{ interpolates } \clz \text{ and } \clw\}.
\]
Then $\lambda_{\clz, \clw} \in (0,1]$, and there exists an extremal function $\vp \in \cle_m(\B^n)$ such that
\[
\lambda_{\clz, \clw} \vp,
\]
interpolates $\clz$ and $\clw$.
\end{corollary}
\begin{proof}
Since $\clz$ and $\clw$ solve the interpolation problem, Theorem \ref{thm: int vs ext} implies that $\clz$ and $\widehat{\clw}$ solve the extremal problem. That is, there exists $\vp \in \cle_m(\B^n)$ such that
\[
\vp(z_i) = \widehat{w}_i,
\]
for all $i=1, \ldots, m$. Recall that
\[
\widehat{w}_j = \frac{1}{\|\Pi_{\clz, {\clw}}\|} w_j,
\]
for all $j=1, \ldots, m$, where $\Pi_{\clz, {\clw}}$ is the functional associated with the data $\clz$ and $\clw$. Recall, by Theorem \ref{thm: NP interpolation contr}, that
\[
\|\Pi_{\clz, {\clw}}\| \leq 1.
\]
Set
\[
\lambda_{\clz, \clw} = \|\Pi_{\clz, {\clw}}\|,
\]
and define
\[
\widetilde{\vp} = \lambda_{\clz, \clw} \varphi.
\]
Clearly, $\widetilde{\vp} \in \cls(\B^n)$ and interpolates the data $\clz$ and $\clw$. Finally, recall from Theorem \ref{thm: norm XQ} the norm expression of $\Pi_{\clz, \clw}$ as
\[
\|\Pi_{\clz, \clw}\| = \inf \{\|\theta\|_\infty: \theta \in \cls(\B^n), S_\theta = X_{\clz, \clw}\}.
\]
On the other hand, by \eqref{eqn: S vp = X ZW}, for every $\theta \in \cls(\B^n)$,
\[
S_\theta = X_{\clz, \clw},
\]
if and only if $\theta$ interpolates $\clz$ and $\clw$. Consequently,
\[
\|\Pi_{\clz, \clw}\|  = \inf \{\|\eta\|_\infty: \eta \in \cls(\B^n) \text{ interpolates } \clz \text{ and } \clw\},
\]
and hence
\[
\lambda_{\clz, \clw} = \|\Pi_{\clz, \clw}\|.
\]
Finally, we have
\[
\|\widetilde{\vp}\|_\infty = \|\lambda_{\clz, \clw} \varphi\|_\infty = \lambda_{\clz, \clw} = \|\Pi_{\clz, \clw}\|.
\]
This completes the proof of the corollary.
\end{proof}

Note that $\lambda_{\clz,\clw}\vp$ is a minimum-norm interpolant for the data $\clz$ and $\clw$. The above result is particularly beneficial if one has some description of $\cle_m(\B^n)$. To this end, in Section \ref{Sec: 3 point int}, we will apply this result to a specific case, such as the three-point interpolation problem. Moreover, this result may be viewed as an analogue of the classical interpolation results for the one-variable case. In particular, compare it with Corollary 1.8 in \cite[page 132]{Garnett}.

\newsection{Examples of lifting}\label{Sec: examples}

The goal of this section is to investigate the lifting problem through concrete examples. Specifically, we present examples of contractive module maps that admit a lifting, as well as examples that do not. This allows us, in particular, to apply the preceding results to concrete situations. Recall that an analytic function $f: \B^n \raro \C$ can be expressed as a power series
\[
f(z) = \sum_{\alpha \in \Z_+^n} a_\alpha z^{\alpha},
\]
The notion of degree for polynomials in $\C[z_1, \dots, z_n]$ is already defined at the beginning of Section \ref{sec: Carath interp}. Given $m \in \mathbb{N}$, define finite-dimensional quotient module $\clq_m$ of $H^2(\B^n)$ by (see Section \ref{sec: Carath interp})
\[
\clq_m = \{ p \in \C[z_1, \dots, z_n] : \text{deg } p \leq m\}.
\]
Fix a polynomial $p \in \C[z_1, \dots, z_n]$, and consider the module map
\[
S_p = P_{\clq_m} T_p|_{\clq_m}.
\]
We are interested in the lifting of $S_p$ under the assumption that $\|p\|_2 = 1$.

\begin{theorem}\label{thm: lift poly}
Assume $n > 1$. Let $p \in \C[z_1, \dots, z_n]$, and let $\text{deg }p \leq m$. Suppose
\[
\|p\|_2 = 1.
\]
Then $S_p \in \clb(\clq_m)$ admits a lift if and only if $p$ is a unimodular constant.
\end{theorem}
\begin{proof}
Let $S_p$ admits a lift. We claim that
\[
\|S_p\| = 1.
\]
Indeed, since $S_p$ admits a lift, there exists $\vp \in \cls(\B^n)$ such that
\[
S_p = S_\vp = P_{\clq_m} T_\vp|_{\clq_m},
\]
and hence $\|S_p\| \leq 1$. On the other hand, since $1$ and $p$ are in $\clq_m$, it follows that
\[
S_{p} 1 = P_{\clq_m} p = p,
\]
which implies
\[
\|S_p\| \geq \|S_{p} 1\|_2 = \|p\|_2 = 1.
\]
This proves the claim that $\|S_p\| = 1$. Since $\clq_m$ is finite-dimensional, Theorem \ref{thm: norm 1 no lift} applies, and therefore
\[
S_p= \alpha I_{\clq_m},
\]
for some unimodular constant $\alpha$. The above identity then implies that $p= \alpha$. The converse is obvious.
\end{proof}

For each  $m \geq 1$, define
\[
J_m = \{f \in H^2(\B^n) : f(z) = \sum_{k \in \Z_+^n, |k| \geq m} a_{k} z^k\}.
\]
Equivalently, $J_m$ consists of analytic functions on $\B^n$ whose Taylor series contain no terms of degree less than $m$. The following lemma says, in particular, that $J_m$ is a submodule.

\begin{lemma}\label{J(m) submodule}
$J_{m+1} = \clq_m^{\perp}$ for all $m \geq 1$.
\end{lemma}
\begin{proof}
This is a simple consequence of the fact that $\{z^k: k \in \Z_+^n\}$ forms an orthogonal basis for $H^2(\B^n)$.
\end{proof}

We now recall a perturbation theorem of Rudin \cite[Theorem 4.3]{Rud inner} concerning inner functions on $\B^n$, specialized here to the case of the ball algebra. This result will be applied to construct nontrivial examples.

\begin{theorem}[Rudin]\label{Rudin inner existence}
Let $g$ be a nonzero function in $A(\B^n)$ and let $m \in \mathbb{N}$. Assume that
\[
\|g\|_{\infty} < 1.
\]
Then there exists an inner function $u \in H^\infty(\B^n)$ such that $u$ and $g$ have the same zeros in $\B^n$ and
\[
u - g \in J_m.
\]
\end{theorem}

The following result provides the first application of Rudin's theorem. In particular, it yields inner liftings of certain contractive module maps.

\begin{example}
Fix $m \in \mathbb{N}$, and set
\[
\Lambda  = \{k \in \Z_+^n : |k| \leq m\}.
\]
Choose a finite set of scalars
\[
\{a_k: k\in \Lambda\} \subseteq \mathbb{C}.
\]
Assume that
\[
0 < \sum_{|k| \leq m} |a_k| < 1.
\]
Then
\[
p(z) := \sum_{|k| \leq m} a_k z^k,
\]
is a nonzero polynomial in $\clq_m$. Consider the module map $S_p \in \clb(\clq_m)$. As $p \in \clq_m$, we have
\[
S_p 1 = P_{\clq_m} T_p 1 = p.
\]
On the other hand, for $z \in \B^n$, since
\[
|p(z)| \leq \sum_{|\alpha| \leq m} |a_{\alpha} z^\alpha| \leq \sum_{|\alpha| \leq m} |a_{\alpha}| = \sum_{\alpha \in \Lambda} |a_{\alpha}|,
\]
we conclude that
\[
\|p\|_{H^{\infty}(\B^n)} \leq \sum_{\alpha \in \Lambda} |a_{\alpha}| < 1.
\]
By Rudin, Theorem \ref{Rudin inner existence}, there exists an inner function $u \in H^\infty(\B^n)$ such that
\[
u - p \in J_{m+1}.
\]
In view of Lemma \ref{J(m) submodule}, we know that $J_{m+1} = \clq_m^\perp$. In other words, we have an inner function $u \in H^\infty(\B^n)$ such that
\[
u - p \in \clq_m^\perp.
\]
Corollary \ref{lift lemma} now implies that $S_p$ lifts to $S_u$.
\end{example}

It is interesting to observe how two seemingly independent results fit together, namely, the perturbation result of Rudin (Theorem \ref{Rudin inner existence}) and the characterization of lifting along the lines of perturbations in Corollary \ref{lift lemma}.

Let $u \in H^\infty(\B^n)$ be an inner function. Since $u H^2(\B^n)$ is a submodule, it follows that
\[
\clq_u:= H^2(\B^n) \ominus u H^2(\B^n),
\]
is a quotient module. This quotient module is traditionally referred to as a model space \cite{NF}. Recall that (see Section \ref{sec: pert})
\[
\cls\clb(\B^n) = \cls(\B^n) \cap A(\B^n).
\]
In the following, we give a criterion of lifting to ball algebra. We remark that this result holds specifically for $n>1$.

\begin{theorem}\label{inner ball}
Let $n > 1$, $u \in H^{\infty}(\B^n)$ be a nonconstant inner function, and let $X \in \clb_1(\clq_u)$ be a module map. Assume that
\[
\psi:= X (P_{\clq_u}1) \in A(\B^n).
\]
Then $X$ admits a lift to $\cls\clb(\B^n)$ if and only if
\[
\|\psi\|_{\infty} \leq 1.
\]
Moreover, in this case, the lifting is unique.
\end{theorem}
\begin{proof}
Suppose $X$ lifts to $A(\B^n)$. By Corollary \ref{lemma: lift ball algbera}, there exists $\vp \in A(\B^n) \cap \cls(\B^n)$ such that
\[
\psi - \vp \in \clq_u^{\perp} = uH^2(\B^n).
\]
Since $\vp$ and $\psi$ are in $A(\B^n)$, it follows that $\psi - \vp \in A(\B^n)$, and hence
\[
\psi - \vp \in A(\B^n) \cap u H^2(\B^n).
\]
Assume, for contradiction, that $\psi \neq \vp$. From the above, there exists a nonzero $g \in H^2(\B^n)$ such that
\[
\psi - \vp = u g.
\]
However, this is impossible, in view of \eqref{eqn: Rudin Smirnov} (which is again a result of Rudin; see \cite[Theorem 4.6]{Rud inner}). This contradiction shows that $\psi = \vp$. Consequently, $\|\psi\|_{\infty} = \|\vp\|_{\infty} \leq 1$. The converse is immediate by Corollary \ref{lemma: lift ball algbera}.
\end{proof}

To better illustrate the theorem, we now construct a concrete example satisfying its hypotheses.

\begin{example}\label{ball lift example}
Fix $m \in \mathbb{N}$ and multi-index $k \in \Z_{+}^n$ such that
\[
|k| = m+1.
\]
Consider the polynomial
\[
p(z) = \frac{1}{2} z^k \in H^2(\B^n).
\]
By Rudin, Theorem \ref{Rudin inner existence}, there exists a non-constant inner function $u \in H^{\infty}(\B^n)$ such that
\[
u - p \in J_{m+1} = \clq_m^\perp.
\]
By our choice, we have $p \in \clq_m^\perp$, and hence
\[
u = p + (u - p) \in \clq_m^\perp.
\]
This implies $u H^2(\B^n) \subseteq \clq_m ^\perp$, equivalently, $\clq_m \subseteq \clq_u$. Let $q \in \C[z_1, \dots , z_n]$ with $\text{deg }q \leq m$, and consider the module map $X = S_{q} \in \clb(\clq_u)$ defined by
\[
S_{q} f = P_{\clq_u} q f
\]
Then,
\[
XP_{\clq_u}1 = S_{q} P_{\clq_u} 1 = P_{\clq_u} T_q P_{\clq_u} 1 = q \in A(\B^n).
\]
By Theorem \ref{inner ball}, $X$ admits a lift to $A(\B^n)$ if and only if
\[
\|q\|_{\infty} \leq 1.
\]
\end{example}

The existence and structural properties of inner functions have played a fundamental role throughout this paper. We conclude this section with yet another application of this existence.

As pointed out, quotient modules of $H^2(\B^n)$ for $n > 1$ are, by far, intricate objects. In what follows, we focus on a particular class of module maps that admit liftings. Remarkably, these liftings exhibit a rigidity phenomenon, which we discuss after the theorem.

We remark that any nonconstant inner function is contained in some nontrivial quotient module. To see this, fix a nonconstant inner function $\varphi \in H^\infty(\B^n)$. Consider the submodule
\[
\cls = \varphi H^2_0(\B^n),
\]
where $H^2_0(\B^n) = \{f \in H^2(\B^n): f(0) = 0\}$. Let $\clq = \cls^\perp$. Then $\clq$ is a quotient module that contains $\varphi$. Indeed, for $f \in H^2_0(\B^n)$, we have
\[
\la \varphi, \varphi f \ra = \la 1, f \ra = \overline{f(0)} = 0.
\]
The above argument ensures that the hypothesis in the following theorem is non-vacuous.

\begin{theorem}\label{lifiting inner multiple}
Let $\clq \subseteq H^2(\B^n)$ be a quotient module containing a nonconstant inner function $\vp \in \cls(\B^n)$, and let $\lambda \in \C$. Then the module map
\[
S_{\lambda \vp} = P_\clq T_{\lambda \vp}|_\clq,
\]
on $\clq$ admits a lift if and only if
\[
|\lambda| \leq 1.
\]
\end{theorem}
\begin{proof}
Following Theorem \ref{contr lift}, we write $X = S_{\lambda \vp}$ and $\psi = X (P_{\clq} 1)$. We know that $\lambda \vp \in \clq$, and hence
\[
\psi = S_{\lambda \vp} P_{\clq} 1 = P_{\clq}  T_{\lambda \vp} 1 = \lambda \vp.
\]
Then the functional $X_{\clq} : \clm_{\clq} \raro \C$ in Theorem \ref{contr lift} simplifies to
\[
X_{\clq} f = \int_{\Sn} \psi f d\sigma = \int_{\Sn} \lambda \vp f d\sigma = \lambda \int_{\Sn} \vp f d \sigma,
\]
and hence
\[
|X_{\clq} f| = |\lambda| \Big|\int_{\Sn} \vp f d\sigma \Big| \leq |\lambda| \int_{\Sn} |f| d\sigma = |\lambda| \|f\|_{L^1(\Sn)},
\]
for all $f \in \clm_\clq$. On the other hand, since $\vp$ is inner, we have
\[
|X_{\clq} \overline{\vp}| = |\lambda| \int_{\Sn} \vp \bar{\vp} d \sigma  = |\lambda| \int_{\Sn} |\vp|^2 d \sigma = |\lambda|,
\]
and therefore
\[
\|X_{\clq}\| = |\lambda|.
\]
Hence, by Theorem \ref{contr lift}, $X$ admits a lift if and only if $|\lambda| \leq 1$.
\end{proof}

For a quotient module $\clq \subseteq H^2(\B^n)$ and a function $\vp \in \cls(\B^n)$, the module map
\[
X := S_\vp,
\]
always admits a lift, namely $T_\vp$ on $H^2(\B^n)$. However, it may happen that $\vp$ is not a Schur function (that is, $\vp \notin \cls(\B^n)$), but there exists $\psi \in \cls(\B^n)$ such that
\[
X = S_\vp = S_\psi.
\]
Theorem \ref{lifiting inner multiple} shows that, under the stated assumptions, $X = S_{\lambda \vp}$ does not admit the aforementioned property. This rigidity phenomenon is a rather distinctive feature of the lifting theorem of the above class of module maps.

At the end of the proof of Corollary \ref{cor: norm preserving}, we remarked that module maps on infinite-dimensional quotient modules may admit norm-preserving lifts. This stands in contrast to the finite-dimensional case, where Corollary \ref{cor: norm preserving} shows that a module map does not admit a norm-preserving lift whenever $n>1$.

We clarify the infinite-dimensional quotient module aspect. For instance, fix a nonconstant inner function $\varphi$, and choose a quotient module $\clq$ containing $\vp$. Consider the module map $X = S_{\varphi}$ on $\clq$. Clearly $X$ admits a lift, namely $T_\vp$. Moreover, we have
\[
\|X\| = 1 = \|\varphi\|_{\infty}.
\]
On the other hand, since $\varphi$ is inner, it follows from \eqref{eqn: Rudin Smirnov} that
\[
\varphi \notin A(\B^n).
\]
Since $\varphi \in \clq$, it follows that
\[
\clq \nsubseteq A(\B^n).
\]
Consequently, $\clq$ is infinite-dimensional (see the discussion preceding \eqref{eqn: Rudin Smirnov}; see also \cite[Corollary 3.4]{GZ} and  \cite[Chapter 2, Remark 2.5.5]{CG}).

\newsection{Three-point interpolation}\label{Sec: 3 point int}

This section is devoted to some concrete examples of interpolation problems. An interpolation problem is said to be a \textit{three-point interpolation problem} if the interpolation data set $\clz$ consists of three points: $\clz = \{z_1, z_2, z_3\} \subset \B^n$. As usual, in this case, we write the collection of target values as $\clw = \{w_1, w_2, w_3\} \subset \D$. We begin by illustrating a solution to a three-point interpolation problem.

Define
\[
\clf_3(\B^n) = \clf_{D} \cup \clf_{ND},
\]
where
\[
\clf_{D}= \Big\{x \mapsto \frac{2 x_1(1-\tau x_1) - \overline{\tau} \omega^2 x_2^2}{2(1- \tau x_1) - \omega^2 x_2^2}: |\tau| = 1, |\omega| \leq 1 \Big\},
\]
and
\[
\clf_{ND}= \Big\{x \mapsto \frac{x_1^2}{2 - a^2} + \frac{2 \sqrt{1 - a^2} x_2}{2 - a^2} : a \in [0, 1) \Big\}.
\]
In the above, we represented $x$ as $(x_1, \dots ,x_n)$. In \cite[Theorem 2]{K}, Kosi\'nski and Zwonek proved the following result: If the three-point Pick interpolation problem is extremal, then, up to a composition with automorphisms of $\B^n$ and $\D$, it is interpolated by a function belonging to $\clf_3(\B^n)$.

The result of Kosi\'nski and Zwonek, together with Corollary \ref{cor: int vs ext}, yields a concrete representation of at least one solution to a given three-point interpolation problem:

\begin{theorem}\label{thm: 3point rational}
Suppose $\clz = \{z_1, z_2, z_3\} \subset \B^n$ and $\clw = \{w_1, w_2, w_3\} \subset \D$ are interpolation data, with the $w_i$ not all equal. If $\clz$ and $\clw$ solve the interpolation problem, then there exist a scalar $\lambda \in (0,1]$ and extremal function $\vp$ such that
\[
\lambda \vp,
\]
solve the interpolation problem for $\clz$ and $\clw$. Moreover, up to automorphisms of $\B^n$ and $\D$, the function $\vp$ may be chosen from
\[
\vp \in \clf_3(\B^n),
\]
and the scalar $\lambda$ may be chosen to satisfy the minimality condition
\[
\lambda
=
\inf\bigl\{\|\eta\|_\infty : \eta \in \cls(\B^n)
\text{ interpolates } \clz \text{ and } \clw\bigr\}.
\]
\end{theorem}

The final part of the above result follows the same lines of proof of Corollary \ref{cor: int vs ext}. Now we consider a definite $3$-point interpolation problem.

\begin{example}\label{ex: 3 point}
Let $n > 2$. Denote by $\{e_j\}_{j=1}^n$ the standard orthonormal basis of $\C^n$, and define
\[
\lambda = \sqrt{1 - \Big(\frac{3}{8}\Big)^{\frac{1}{n}}},
\]
and
\[
z_j = \lambda e_j,
\]
for all $j=1,2,3$. Let
\[
M = \max \{\|\Sk(\cdot, z_j)\|_{\infty}: j =1,2,3\}.
\]
Note that $M > 1$. Finally, define
\[
\begin{cases}
w_1 = \frac{1}{2M},
\\
w_2 = \frac{1}{3M},
\\
w_3 = \frac{1}{4M}.
\end{cases}
\]
We now consider the three-point interpolation problem corresponding to the data $\clz = \{z_1, z_2, z_3\}$ and $\clw = \{w_1, w_2, w_3\}$. As in Theorem \ref{thm: NP interpolation contr}, we let
\[
\psi_{\clz, \clw} = c_1 \Sk(\cdot, z_1) + c_2 \Sk(\cdot, z_2) + c_3 \Sk(\cdot, z_3),
\]
where the scalars $c_1, c_2$, and $c_3$ are given by
\[
\begin{bmatrix}
c_1
\\
c_2
\\
c_3
\end{bmatrix}
=
\begin{bmatrix}
\Sk(z_i, z_j)
\end{bmatrix}_{3 \times 3}^{-1}
\begin{bmatrix}
w_1
\\
w_2
\\
w_3
\end{bmatrix}.
\]
Since $\la z_i, z_j \ra = \lambda^2 \delta_{ij}$, and $\Sk(z, w) = \frac{1}{(1 - \la z, w\ra)^n}$ for all $z, w \in \B^n$, it follows that
\[
\Sk(z_t, z_t) = \frac{1}{(1 - \la z_t, z_t \ra)^n} = \frac{1}{(1 - \lambda^2)^n},
\]
for all $t = 1, 2,3$, and
\[
\Sk(z_i, z_j) = 1,
\]
for all $i \neq j$.  Therefore,
\[
[\Sk(z_i, z_j)]_{3 \times 3} =
\begin{bmatrix}
a & 1 & 1
\\
1 & a & 1
\\
1 & 1 & a
\end{bmatrix},
\]
where
\[
a = \frac{1}{(1 - \lambda^2)^n}.
\]
Computing the inverse of the matrix $[\Sk(z_i, z_j)]_{3 \times 3}$, we obtain
\[
\begin{bmatrix}
c_1
\\
c_2
\\
c_3
\end{bmatrix}
= \frac{1}{(a - 1)(a + 2)}
\begin{bmatrix}
(a + 1)w_1 - w_2 - w_3
\\
-w_1 + w_2(a + 1) - w_3
\\
-w_1 - w_2 + (a + 1)w_3
\end{bmatrix}.
\]
This further implies that $c_1, c_2, c_3 > 0$. Now
\[
\begin{split}
\|\psi_{\clz, \clw}\|_{\infty} & \leq c_1 \|\Sk(\cdot, z_1)\|_{\infty} + c_2 \|\Sk(\cdot, z_2)\|_{\infty} + c_3 \|\Sk(\cdot, z_3)\|_{\infty}
\\
& \leq M(c_1 + c_2 + c_3)
\\
& \leq M \frac{1}{(a + 2)} (w_1 + w_2 + w_3)
\\
& = \frac{13}{12(a + 2)}
\\
& < 1,
\end{split}
\]
that is, $\|\psi_{\clz, \clw}\|_{\infty} < 1$. By \eqref{eqn: basic int fn}, the interpolation problem has a solution.
\end{example}

\section{Concluding remarks}\label{sec: remark}

Approximation, interpolation and lifting are central themes in the theory of Hilbert function spaces and are also closely connected to several applied areas of mathematics, including electrical engineering (cf. Helton \cite{Hel1, Hel1}). Results in these directions contribute to the understanding of a wide range of problems, such as the structure of submodules and quotient modules; the theory of linear operators; complex variables; neural networks; and techniques involving reproducing kernel Hilbert spaces, including recent developments in machine learning \cite{Baek, OG, Park, Wer}. It is also important to note that for $n=1$, Schur functions admit useful fractional linear-type representations, also known as realization formulas \cite{AM}. However, various complications arise when attempting to represent Schur functions in higher dimensions. Another fundamental problem concerns the structure of submodules and quotient modules of $H^2(\B^n)$. We hope that the results presented in this paper will contribute to the further development and understanding of these areas. We refer the reader to \cite{AY, Nicolau 2, Nick} for the one-variable theory and its impact on the general theory.

In the passage from one-variable to multivariable Hilbert function spaces, the two most natural domains are the polydisc and the unit ball. The paper \cite{KD} resolved the lifting and interpolation problems on $\D^n$. Some of the qualitative and quantitative features established there also arise in our study of $\B^n$. With this comparison in mind, we make the following general observation:

\begin{enumerate}
\item The construction of the function $\psi_{\clz, \clw}$, as well as the qualitative and quantitative criteria established in this paper, was also obtained in the polydisc setting in \cite{KD}. In the present work, we prove these implications using a combination of techniques; some new and others carefully adapted, with significant modifications, from $\D^n$ setting. A fundamental difference in the ball setting is the absence of a convenient orthonormal basis for $L^2(\Sn)$; in contrast, for $L^2(\mathbb{T}^n)$, the monomials form an orthonormal basis. This feature played a central role in the interpolation and lifting results for $\D^n$ in \cite{KD}, and its absence highlights a key technical challenge in the present context.

\item While working on $\B^n$ (as well as on $\D^n$ \cite{KD}), we apply substantially finer techniques that allow us to avoid many of the subtleties inherent in several-variable function theory. On the one hand, such refined tools are essential for eliminating case-by-case analyses and obtaining a general result. On the other hand, the results also suggest that, when one focuses on specific domains such as $\B^n$ or $\D^n$, there remains considerable scope for sharpening and extending the results presented here, as well as those in the earlier work \cite{KD}, to more specialized settings, such as specific quotient modules or module maps.

\item While working on $\B^n$, we found that results concerning Schur functions and Hilbert function spaces are often considerably more involved than in other standard domains, such as $\D^n$. This may be partly due to the lack of explicit examples of inner functions on $\B^n$. Indeed, as pointed out earlier, many of the results, examples, and counterexamples in this paper rely crucially on the structure of inner functions on $\B^n$.
\end{enumerate}

Before closing, we outline some simple yet useful general observations. First, we note that every module map on a finite-dimensional quotient module admits a lift to a function in $H^\infty(\B^n)$. The crucial point underlying the proof is that these quotient modules are naturally realized inside the ball algebra as noted in \eqref{eqn: fd Q in A}.

\begin{proposition}
Let $\clq$ be a finite-dimensional quotient module of $H^2(\B^n)$, and let $X \in \clb(\clq)$ be a module map. Then there exists $\vp \in H^\infty(\B^n)$ such that
\[
X = S_\vp.
\]
Moreover, $\vp$ can be chosen as
\[
\vp = X(P_\clq 1),
\]
and therefore,
\[
X = S_{X(P_{\clq}1)}.
\]
\end{proposition}
\begin{proof}
As $\clq \subseteq H^2(\B^n)$ is finite-dimensional, it follows from \eqref{eqn: fd Q in A} (which comes from \cite[Corollary 3.4]{GZ} and \cite[Chapter 2, Remark 2.5.5]{CG}) that
\[
\clq \subset A(\B^n) \subset H^{\infty}(\B^n).
\]
In particular, we have
\[
\psi := X(P_{\clq} 1) \in \clq \subseteq A(\B^n),
\]
that is, $\psi \in H^{\infty}(\B^n)$. If we write $0$ as
\[
\psi - \psi = 0 \in \clq^\perp,
\]
then Theorem \ref{lemma: lift ball 1} implies that $X = S_{\psi}$ and completes the proof of the lemma.
\end{proof}

Throughout the paper, the space of mixed functions $\clm(\Sn)$ has played an important role. We now discuss a decomposition of this space that aligns with the homogeneous decomposition studied by Rudin \cite{Rud}. Recall that
\[
\clm(\Sn) = L^2(\Sn) \ominus [H^2(\Sn) \dot+ H^2(\Sn)^{conj}].
\]
For each $p, q \in \Z_+$, define $H(p,q)$ as the vector space of all harmonic homogeneous polynomials on $\C^n$ that have total degree $p$ in the variables $z_1, \ldots, z_n$, and total degree $q$ in the variables $\overline{z_1}, \ldots, \overline{z_n}$. This notion is related to spherical harmonic functions, as discussed in \cite[Section 12.1]{Rud}. In view of the identification
\[
f \leftarrow\raro f|_{\Sn},
\]
one can represent $H(p, q)$ as a vector subspaces of $L^2(\Sn)$ for all $p,q \in \Z_+$. Moreover, by \cite[Theorem 12.2.3]{Rud}, it follows that $L^2(\Sn)$ is the direct sum of the pairwise orthogonal spaces $H(p,q)$, $p,q \in \Z_+$:
\[
L^2(\Sn) = \bigoplus_{p,q \in \Z_+} H(p, q).
\]
Observe $H(p, 0)$ consists of holomorphic homogeneous polynomials of degree $p$ for all $p \in \Z_+$. Moreover
\[
H^2(\Sn) = \bigoplus_{p \in \Z_+} H(p, 0)
\]
Now, by taking conjugates, it follows that
\[
H^2(\Sn)^{conj} = \bigoplus_{q \in \Z_+} H(0, q)
\]
Using these direct sum decompositions, along with that of $L^2(\Sn)$, we finally conclude that
\[
\clm(\Sn) = \bigoplus_{p, q \geq 1} H(p, q)
\]

\begin{remark}
The above description of $\clm(\Sn)$ can be used to provide a slightly different proof of Theorem \ref{contr lift}. However, the proof presented there also has the potential to work in more general cases, such as weighted Bergman spaces.
\end{remark}

Alongside inner functions, extremal functions constitute one of the central themes of this paper. We conclude by emphasizing some compelling reasons why the class of extremal functions deserves special attention. Recall that $\cle(\B^n)$ denotes the class of extremal functions, while $\cle(\D)$ is the set of all finite Blaschke products.

\begin{enumerate}
\item Unlike the $n=1$ case, there is no concrete description or example of inner functions on $\B^n$, and in particular there is no satisfactory analogue or example of finite Blaschke products whenever $n>1$.

\item The class $\cle(\B^n)$ contains many explicit examples that can be constructed via extremal interpolation problems (cf. \cite[Theorem 2]{K}).

\item The Carath\'{e}odory approximation theorem and Pick's theorem on $\B^n$, established in this paper, along with the striking parallels between inner and extremal functions also outlined here, suggest that extremal functions occupy a central place in function theory.
\end{enumerate}

These observations suggest that extremal functions merit a systematic investigation. One fundamental question is to understand their structure and obtain representations, at least for rational extremal functions. From this perspective, it is tempting to conjecture that every extremal function on $\B^n$, for $n>1$, is rational.

\vspace{0.2in}

\noindent\textbf{Acknowledgement:}
The third named author is supported in part by MATRICS grant (ANRF/ARGM/2025/000130/MTR) by ANRF, Department of Science \& Technology (DST), Government of India.

\bigskip

\bibliographystyle{amsplain}

\end{document}